\documentclass[12pt, a4paper]{amsart}
\usepackage{amsmath}
\usepackage{geometry,amsthm,graphics,tabularx,amssymb,shapepar}
\usepackage{fancybox,color,epsf,slidesec,epic,eepic,amsmath,amsfonts,amssymb}
\usepackage{amscd}
\usepackage{mathrsfs}
\usepackage{epsfig}
\usepackage[all,cmtip]{xy}

\usepackage{graphicx}
\usepackage{aurical}
\usepackage{url}
\usepackage{youngtab}
\usepackage[small,nohug,heads=vee]{diagrams}

\newcommand{\CK}{{\mathcal {K}}}

\newcommand{\CM}{{\mathcal {M}}}
\newcommand{\CN}{{\mathcal {N}}}
\newcommand{\CO}{{\mathcal {O}}}

\newcommand{\CR}{{\mathcal {R}}}

\newcommand{\CX}{{\mathcal {X}}}
\newcommand{\CY}{{\mathcal {Y}}}

\newcommand{\SY}{\mathscr{Y}}
\renewcommand{\SS}{\mathscr{S}}
\newcommand{\SV}{\mathscr{V}}
\newcommand{\SA}{\mathscr{A}}
\newcommand{\sO}{\mathscr{O}}

\newcommand{\rD}{{\mathrm {D}}}

\newcommand{\rF}{{\mathrm {F}}}

\newcommand{\rH}{{\mathrm {H}}}

\newcommand{\Ad}{{\mathrm{Ad}}}

\newcommand{\Ch}{{\mathrm{Ch}}}

\newcommand{\GL}{{\mathrm{GL}}}

\newcommand{\Hom}{{\mathrm{Hom}}}

\renewcommand{\Im}{{\mathrm{Im}}}

\newcommand{\Ker}{{\mathrm{Ker}}}

\newcommand{\SO}{{\mathrm{SO}}}

\newcommand{\Sp}{{\mathrm{Sp}}}
\newcommand{\Stab}{{\mathrm{Stab}}}
\newcommand{\Span}{{\mathrm{Span}}}

\newcommand{\GNM}{{\mathrm{GNM}}}

\newcommand{\wt}{\widetilde}

\newcommand{\cover}[1]{\widetilde{#1}}

\newcommand{\vsp}{{\vspace{0.2in}}}

\newcommand{\Irr}{\operatorname{Irr}}

\newcommand{\ad}{\operatorname{ad}}

\newcommand{\Tr}{\operatorname{Tr}}

\newcommand{\oG}{\operatorname{G}}

\newcommand{\g}{\mathfrak g}

\renewcommand{\k}{\mathfrak k}

\newcommand{\p}{\mathfrak p}

\newcommand{\n}{\mathfrak n}
\renewcommand{\u}{\mathfrak u}

\newcommand{\m}{\mathfrak m}

\renewcommand{\sl}{\mathfrak s \mathfrak l}

\newcommand{\Z}{\mathbb{Z}}
\newcommand{\C}{\mathbb{C}}
\newcommand{\R}{\mathbb R}

\newcommand{\DD}{\mathcal D}

\newcommand{\la}{\langle}
\newcommand{\ra}{\rangle}

\newcommand{\be}{\begin {equation}}
\newcommand{\ee}{\end {equation}}
\newcommand{\bee}{\begin {equation*}}
\newcommand{\eee}{\end {equation*}}

\newcommand{\bfee}{{\mbox{\bf e}}}
\newcommand{\Orb}{\mathcal{O}}
\newcommand{\Wh}{\operatorname{Wh}}
\newcommand{\Gen}{\operatorname{Gen}}
\newcommand{\Sch}{\mathscr{C}}
\def\DD{\nabla}
\def\gDD{\nabla^{\mathrm{gen}}}

\newcommand{\set}[2]{ \left\{ {#1} \, \left| \, {#2} \right\}\right.}

\newcommand{\bfV}{\mathbf{V}}
\newcommand{\bfS}{\mathbf{S}}
\newcommand{\bfM}{\mathbf{M}}
\newcommand{\bfW}{\mathbf{W}}
\newcommand{\bfWo}{\mathbf{W}^{\circ}}
\newcommand{\bfG}{\mathbf{G}}
\newcommand{\bfK}{\mathbf{K}}
\newcommand{\Nil}{\mathrm{Nil}}
\newcommand{\AV}{\mathrm{AV}}
\newcommand{\bcO}{\overline{\mathcal{O}}}

\DeclareMathOperator{\Ann}{Ann}
\newcommand{\wtbfK}{\widetilde{\mathbf{K}}}
\newcommand{\wtK}{\widetilde{{K}}}
\newcommand{\wtG}{\widetilde{{G}}}
\def\Thetav{\check{\Theta}}

\newcommand{\oliftc}{\boldsymbol{\theta}}

\def\NilGC{\Nil_{\bfG}(\g)}
\def\NilGCp{\Nil_{\bfG'}(\g')}

\def\mktvvp{\varsigma_{{V},{V}'}}

\def\DD{\nabla}
\def\DDc{\boldsymbol{\nabla}}
\def\gDD{\nabla^{\mathrm{gen}}}

\def\dliftv{\vartheta}
\def\Thetab{\bar{\Theta}}
\def\MM{\bfM}
\def\MMP{M}
\def\Xo{\CX^\circ}
\def\Stab{\mathrm{Stab}}

\def\bfii{\mathbf{i}}

\theoremstyle{Theorem}

\theoremstyle{plain}
\newtheorem{thm}{Theorem}[section]

\newtheorem{lem}[thm]{Lemma}
\newtheorem{prop}[thm]{Proposition}

\newtheorem{defn}[thm]{Definition}
\newtheorem{rmk}[thm]{Remark}

\title[Local theta correspondence and nilpotent invariants]{Local theta correspondence and nilpotent invariants}

\author{Chen-Bo Zhu}
\address{Department of Mathematics\\
National University of Singapore\\
Block S17, 10 Lower Kent Ridge Road, Singapore 119076}
\email{matzhucb@nus.edu.sg}

\begin{document}


\dedicatory{Dedicated to Joseph Bernstein on the occasion of his seventy-second birthday}


\begin{abstract} We consider two types of nilpotent invariants associated to smooth representations, namely generalized Whittaker models, and associated characters (in the case of a real reductive group). We survey some recent results on the behavior of these nilpotent invariants under local theta correspondence, and highlight the special role of a certain double fiberation of moment maps.

\end{abstract}

\maketitle

\section{Introduction}
\label{sec:0}
One of the most fruitful approaches in representation theory is to understand representations through their invariants. Among the most well-known invariants are matrix coefficients (and their growth), and the character of a representation, which are of analytic nature. In the case of a real reductive group, fundamental invariants of algebraic nature include the infinitesimal character and the restriction of the representation to a maximal compact subgroup. Each of the afore-mentioned invariants has played a prominent role in the most important tasks of representation theory, for example in the classification of representations \cite{La,Vo79}.

In the current article, we will be concerned with smooth representations of a reductive group defined over a local field $\rF$ of characteristic $0$. By a smooth representation we mean a smooth representation in the usual sense for $\rF$ non-Archimedean, namely it is locally constant, and a Casselman-Wallach representation for $\rF$ Archimedean \cite{Ca89, Wa2}.

The invariants we will consider will be informally referred to as nilpotent invariants as they depend on various types of nilpotent orbits and thus reflect various kinds of singularities.

The first nilpotent invariant to be examined is the space of generalized Whittaker models \cite{Ya86,MW87,WaJI}. As it is well-known, the study of Whittaker and generalized Whittaker models for representations evolved in connection with the theory of automorphic forms, and has found numerous applications. See for example \cite{Sh74,NPS73,Ko78,Ka85,Ya86, WaJI}. We will recall the basic definitions regarding generalized Whittaker models in Section \ref{subsec:GWM}.

The second nilpotent invariant to be examined is defined by Vogan \cite{Vo89} for a Harish-Chandra $(\g,K)$-module (and therefore for a Casselman-Wallach representation of a real reductive group $G$), and is called the associated character. This is an element of the Grothendieck group of $\bfK$-equivariant coherent sheaves on $\CN (\p)$ (the nilpotent cone on $\p$), and is a refinement of the associated cycle, which is a (finite) union of the closure of $\bfK$-orbits in $\CN(\p)$ with multiplicities. Here $\g=\mathfrak{k}\oplus \mathfrak{p}$ is the complexified Cartan decomposition, and $\bfK$ is the complexification of $K$.
We will recall the basic definitions regarding the associted character in Section \ref{subsec:AC}.

\begin{rmk} Another well-known nilpotent invariant (which we will not examine) is the wave front cycle,
defined by Harish-Chandra in the non-Archimedean case  (\cite{HC78}) and by Howe and Barbasch-Vogan in the Archimedean case (\cite{HoWF,BV80}; see also
\cite{Ro95}). In the non-Archimedean case, M\oe{}glin and Waldspurger have established \cite{MW87} that the wave front cycle controls the spaces of generalized Whittaker models.
In the Archimedean case, Schmid and Vilonen have shown \cite{SV00} that the wave front cycle and the associate cycle determine each other through the Kostant-Sekiguchi correspondence \cite{Se87}.
\end{rmk}

\begin{rmk} While both of nilpotent invariants examined in this paper are defined for a smooth representation, it is generally very difficult to compute them, even when we have a precise description of the representation.
\end{rmk}

\vsp

We now come to the other part of the title regarding local theta correspondence. Let $(G,G')\subseteq \Sp (W)$ be a reductive dual pair of type I in the sense of Howe \cite{Ho79}.
By fixing a non-trivial unitary character $\psi$ of $\rF$, we have the smooth oscillator representation of $\cover{\Sp}(W)$ associated to $\psi$. For a smooth irreducible genuine representation $\pi$ of
$\cover{G}$, we have the full theta lift $\Theta(\pi)$ of $\pi$, which is a smooth representation of $\cover{G'}$ of finite length. For readers' convenience, we have included a brief review of the theory of dual pairs and local theta correspondence in Section \ref{sec:DPreview}. The question we wish to address in the current article is the behavior of these invariants under local theta correspondence, more specifically the relationship of the nilpotent invariants of $\Theta (\pi)$ with those of $\pi$.

In the first part of this article (Section \ref{sec:GWM}), we focus on the behavior of the generalized Whittaker models under local theta correspondence. We shall summarize the result as follows. Let $\CO$ be a nilpotent orbit in the Lie algebra $\g$ of $G$. Let $\gamma=\{H,X,Y\}$ be an $\sl_2$-triple with $X\in \CO$ and $P=MN$ the associated parabolic subgroup of $G$. Then we have an oscillator representation or a $1$-dimensional character $\SS_{{\gamma}}$ of $\cover{M_X}\ltimes N$, where $\cover{M_X}$ is a double cover of the stabilizer group $M_X$ of $X$ in $M$. For a reductive subgroup $R$ of $M_X$ and a genuine representation $\tau$ of $\cover{R}$, we consider its model
\begin{equation*}
\Wh_{\Orb,\tau}(\pi)=\Hom_{\cover{R}\ltimes N}(\mathscr{V},\mathscr{V}_{\tau}\widehat \otimes
\SS_{{\gamma}}).
\end{equation*}
(When $R$ is the trivial group, we will simply write $\Wh_{\Orb}(\pi)$.)

To any nilpotent orbit $\CO'\subset \g'$ in the image of the moment map, one may associate a nilpotent orbit $\CO \subset \g$ via the moment maps, called the generalized descent of $\CO'$. The associated stabilizer groups $M_X$ and $M'_{X'}$ are of the form
\begin{equation*}
M_{X} \supset M_{X, X'}\times L, \qquad M'_{X'}=M_{X, X'}\times L',
\end{equation*}
for some reductive dual pair $(L,L')$ of the same type as $(G,G')$. Then the main result (which generalizes \cite{GZ}, in the so-called descent case where $L$ is trivial and $M_X=M_{X,X'}$) says that
\[\Wh_{\Orb',\tau '}(\Theta (\pi ))\simeq \Wh_{\Orb, \Theta (\tau')^{\vee}}(\pi ^{\vee})\]
as $\cover{M_{X, X'}}$-modules.  Here $\Theta (\tau')$ is the full theta lift of $\tau '$ with respect to the dual pair $(L,L')$, and the symbol $\vee$ indicates the dual in the category of Casselman-Wallach representations. We also prove that if $\Orb'$ is not in the image of the moment map, then
\[\Wh_{\Orb'}(\Theta (\pi))=0.\]
For precise statements and unexplained notations, see Sections \ref{subsec:GD} and \ref{subsec:Trans}.

\vsp
In the second part of this article (Section \ref{sec:AC}), we focus on the behavior of the associated character under local theta correspondence (over $\R$). We shall also summarize the result. Let $\CO$ be a nilpotent orbit in $\g$, now the complexified Lie algebra of $G$. To any $(\g,\cover{K})$ module $\Pi$ of finite length whose complex associated variety is contained in the closure of $\CO$, Vogan associates an element (called the associated character)
\[\Ch_{\CO}(\Pi)\in \CK_{\CO}(\wt{\bfK}):=\bigoplus
  \mathrm \CK_{\sO}(\wt{\bfK})\]
  where $\sO$ runs over $\bfK$-orbits in $\CO\cap \p$, and $\CK_{\sO}(\wt{\bfK})$ is the Grothendick group of $\wt{\bfK}$-equivariant algebraic vector bundles on $\sO$.

Suppose that nilpotent orbits $\CO\subset \g$ and $\CO'\subset \g'$ are the images of an element of maximal rank under the moment maps. Then the result of \cite{MSZ} gives an upper bound
 \begin{equation*}
    \Ch_{\CO'}(\Theta (\pi ^{\vee}))\preceq \dliftv_{\CO,\CO'}(\Ch_{\CO}(\pi )),
  \end{equation*}
  where $\dliftv_{\CO,\CO'}$ is a certain sum of maps between the Grothendieck groups of algebraic vector bundles and is also defined via the moment maps. Moreover, if $(\pi, G')$ is in the so-called convergent range (see Definition \ref{defn:CR}), then this inequality becomes an equality with $\Theta (\pi^{\vee})$ replaced by a certain quotient $\Thetab (\pi)$ of $\Theta (\pi^{\vee})$.  We refer the reader to Sections \ref{subsec:lift.AC} and \ref{subsec:UBEQ} for precise statements and unexplained notations.

Our results have the following immediate consequences: if $\Wh_{\Orb, \Theta (\tau')^{\vee}}(\pi ^{\vee})\ne 0$ for some $\Orb$ and $\tau '$, then $\Theta (\pi)\ne 0$. For example if $(G,G')$ is in the so-called stable range with $G$ the smaller member \cite{Li89}, one may find a nilpotent orbit $\Orb'$ in $\g'$ which descents to the zero orbit in $\g$ (such an $\Orb'$ is called quadratic nilpotent). This immediately implies that $\Theta (\pi)\ne 0$ for any smooth irreducible genuine representation $\pi$ of $\cover{G}$.
Similarly, if $(\pi, G')$ is in the convergent range, and if
$\dliftv_{\CO,\CO'}(\Ch_{\CO}(\pi ))\ne 0$ for some $\CO\subset \g$ and $\CO'\subset \g'$ in the images of an element of maximal rank under the moment maps, then $\Thetab (\pi)\ne 0$ and consequently $\Theta (\pi^{\vee})\ne 0$.

Each of the above provides an effective criterion for the nonvanishing of $\Theta (\pi)$. Both are internal in the sense that one does not invoke other dual pairs or other representations. Of course conservation relations (conjectured by Kudla and Rallis \cite{KR2} and established by B. Sun and the author in its full generality \cite{SZ}) provide one other way to determine the vanishing/nonvanishing, which is external. (We refer the reader to \cite{SZ2} for a easy read on the conservation relations.) Note that vanish/non-vanishing of $\Theta (\pi)$ amounts to an inequality on the first occurrence index of $\pi$ in a (generalized) Witt tower, and so will yield information on the non-vanishing/vanishing in another (generalized) Witt tower. The two results on the correspondence of nilpotent invariants surveyed in this article, when combined with conservation relations, will thus provide additional information on the questions of vanishing and nonvanishing in local theta correspondence.

As a concluding remark, the author wishes to highlight in this article the special role of the double fiberation of moment maps in questions on the local theta correspondence. Important earlier work on this double fiberation of moment maps include those of Kraft-Procesi \cite{KP82} and Prezbinda \cite{Pz}. More recent work connected to the theme of this article include \cite{NOTYK,NZ,LM,GZ,MSZ}.

\vsp
\noindent {\bf Acknowledgements}: The author would like to thank the anonymous referee for his thoughtful comments and suggestions, which helped to substantially improve the exposition of this article.  The author would also like to thank Kyo Nishiyama whose collaboration a decade ago has helped to shape the author's thinking on nilpotent invariants of smooth representations.

\section{Dual pairs and local theta correspondence: a brief review}
\label{sec:DPreview}
We review the basic set-up of the theory of dual pairs \cite{Ho79}.

Let $\rF$ be a local field of characteristic zero, and let $\rD$ be one of the following division algebras over $\rF$:
the field $\rF$ itself, a quadratic field extension of $\rF$ or the
central quaternion division algebra over $\rF$.

Let $\epsilon=\pm 1$, and $V$ be an $\epsilon$-Hermitian space, namely a finite dimensional right $\rD$-vector space, equipped with the $\epsilon$-Hermitian form
$(\cdot , \cdot  )_{V}$. The isometry group $G=\oG(V)$ is a type I classical group, as in the following table.
\vsp

\setlength{\tabcolsep}{3pt} 

\begin{tabular}{c|ccc}
$\rD$ & $\rF$ & \ quadratic extension & quaternion algebra\\ 
\hline
$\epsilon=1$ & orthogonal group & unitary group  & quaternionic symplectic group\\ 
\hline 
$\epsilon=-1$ & symplectic group & unitary group  & quaternionic orthogonal group \\
\hline 

\end{tabular}

\vsp

Let $V'$ be an $\epsilon '$-Hermitian right $\rD$-vector space, equipped with the $\epsilon '$-Hermitian form
$(\cdot , \cdot  )_{V'}$, where $\epsilon \epsilon '=-1$. Denote by $G'=\oG(V')$ the corresponding isometry group.
Let $W=:\Hom_{\rD}(V,V')$, and we define a symplectic form $\langle\cdot , \cdot  \rangle _W$ on $W$ by setting
\[
\la T, S\ra_{W}:=\Tr_{\rD/\rF} (T^{\ast}S), \qquad \mbox{$T$, $S\in \Hom_{\rD}(V,V')$,}
\]
where $\Tr_{\rD/\rF} (T^{\ast}S)$ is the trace of $T^{\ast}S$ as an $\rF$-linear transformation, and $T^{\ast}\in \Hom_{D}(V',V)$ is defined by
\[
(Tv,v')_{V'}=(v,T^{\ast}v')_{V}, \qquad \mbox{$v\in V$, $v'\in V'$.}
\]
Let $\Sp (W)$ be the isometry group of $\langle\cdot , \cdot  \rangle _W$.

There is a natural homomorphism: $G \times G'\longrightarrow \Sp(W)$ given by
\[
 (g,g')\cdot T = g' T g^{-1} \qquad \mbox{for $T \in \Hom_{\rD}(V,{V}')$, $g\in G$, $g'\in G'$}.
\]

Fix a non-trivial unitary character $\psi$ of $\rF$, and let $(\omega , \SY )$ be the smooth oscillator representation of $\cover{\Sp}(W)$ associated to $\psi$. Here $\cover{\Sp}(W)$ is the metaplectic cover of $\Sp(W)$, namely the unique topological central extension of the symplectic group $\Sp (W)$ by $\{\pm1\}$ which does not split unless $W=0$ or $\rF=\C$.
For a subgroup $E$ of $\Sp (W)$, let $\cover{E}$ be the inverse image of $E$ in $\cover{\Sp}(W)$. By restriction, $(\omega , \SY)$ yields a representation $
\omega|_{\cover{G}\times \cover{G'}}$, and the basic problem is to understand this restriction.

The main assertion of the theory of local theta correspondence is the Howe duality conjecture \cite{Ho79}: the condition $\Hom_{\cover{G}\times \cover{G'}}(\omega,\pi\widehat \otimes \pi ')\neq 0$
\footnote{Here $\widehat \otimes$ denotes algebraic tensor product for $\rF$ non-Archimedean, and
completed projective tensor product for $\rF$ Archimedean. See \cite{BS} on the issue of algebraic versus smooth correspondences for $\rF$ Archimedean.} determines a one to one correspondence
 \[
\{\pi\in \Irr(\cover{G})\mid \Hom_{\cover{G}}(\omega,\pi)\neq 0\}\smallskip
  \longleftrightarrow \{\pi '\in \Irr(\cover{G'})\mid \Hom_{\cover{G'}}(\omega,\pi ')\neq 0\}.
  \]
Here and as usual, $\Irr (R)$ denotes the set of equivalent classes of irreducible smooth representations of a reductive group $R$.

The Howe duality conjecture is now a theorem in its full generality. It was established by Howe when $\rF$ is Archimedean \cite{Ho89}, by Waldspurger when $\rF$ is $p$-adic (for $p\ne 2$) \cite{Wald}, and completed by Minguez \cite{Mi}, Gan-Takeda \cite{GT} and Gan-Sun \cite{GS}, without any restriction on the residue characteristics.

It is convenient to introduce the following version of the Howe duality theorem: Let $(\pi,\SV)$ be a smooth irreducible genuine representation of
$\cover{G}$, and consider the maximal $\pi$-isotypic quotient of $\SY$, called the Howe quotient of $\pi $. It is of the form
$\SV \widehat \otimes \Theta(\SV)$, where $\Theta(\SV)$ carries a smooth
representation $\Theta(\pi)$ of $\cover{G'}$. The representation $\Theta(\pi)$ is often referred to as the
full theta lift (or colloquially the big theta lift) of $\pi$. The fundamental results of Howe, Waldspurger, Minguez, Gan-Takeda, and Gan-Sun say that $\Theta(\pi)$
is an admissible representation of finite length, moreover, when $\Theta(\pi)$ is non-zero, it has a unique irreducible quotient $\theta(\pi)$, called the (local) theta lift of $\pi$.

\section{Correspondence of generalized Whittaker models}
\label{sec:GWM}

In this section we will discuss a refinement, due to Gomez and the author, of a previous work \cite{GZ}, which computes the generalized Whittaker models of the full theta lift in a general setting. Related earlier works on transition of models include those of Furusawa \cite{Fu95}, M\oe{}glin \cite{Mo98} and Ginzburg-Jiang-Soudry \cite{GJS}. See also Gan's article in this volume \cite{Ga}.

\subsection{Generalized Whittaker models}
\label{subsec:GWM}
We recall some basics of generalized Whittaker models \cite{Ya86,MW87,WaJI}.
Let $G$ be a reductive group
over $\rF$, $\g$ its Lie algebra, on which we fix an
$\Ad\,G$-invariant non-degenerate bilinear form $\kappa$.
Let $\Orb \subset \g$ be a nilpotent orbit, and $X\in \Orb$. We complete $X$ to an $\sl_{2}$-triple $\gamma=\{X,H,Y\}\subset \g$, namely
\[
[H,X]=2X,\qquad [H,Y]=-2Y, \qquad [X,Y]=H.
\]
Set $\g_{j}=\set{Z\in \g}{\ad(H)Z=jZ}$, for $j\in \Z$. Then, from standard
$\sl_{2}$-theory, we have a finite direct sum $\g=\oplus_{j\in \Z}\g_{j}$.
Define the Lie subalgebras $\u=\oplus_{j\leq -2}
\g_{j}$, $\n=\oplus_{j\leq -1} \g_{j}$, $\p=\oplus_{j\leq 0} \g_{j}$
and $\m=\g_{0}$. Let $U$, $N$, $P$, and $M$ be the corresponding
subgroups of $G$. Thus $U=\exp \u$, $N=\exp \n$, $P=\set{p\in G}{\Ad(p)\p
\subset \p}$ and $M=\set{m\in G}{\mbox{$\Ad(m)H=H$}}$. Define the (non-degenerate) character
$\chi_{\gamma}$ of $U$ by
\begin{equation}
\label{defchi}
\chi_{\gamma}(\exp Z)=\psi (\kappa(X,Z)), \ \ \ \ \forall \ Z\in \u.
\end{equation}
Let $M_{X}=\set{m\in M}{\mbox{$\Ad(m)X=X$}}$, the stabilizer group of $X$ in $M$. Then it is well-known
\cite[Section 3.4]{CM92} that
\[
M_{X}=G_{\gamma}:=\set{g\in G}{\mbox{$\Ad(g)X=X$, $\Ad(g)Y=Y$, $\Ad(g)H=H$}}.
\]
In particular $M_{X}$ is reductive. For the moment assume that $\g_{-1}\neq 0$. In this case
$\ad(X)|_{\g_{-1}}:\g_{-1}\longrightarrow \g_{1}$ is an isomorphism,
and we may define a symplectic structure on $\g_{-1}$ by setting
\begin{equation}
\label{defsymg-1}
\kappa_{-1}(S,T)=\kappa(\ad(X)S,T)=\kappa(X,[S,T]), \qquad \mbox{for $S$, $T\in \g_{-1}$}.
\end{equation}
We may exhibit a canonical surjective group homomorphism from $N$ to the associated Heisenberg group $\rH_{\g_{-1}}$ which maps $\exp Z$ to $\kappa(X,Z)$ in the center of $\rH_{\g_{-1}}$, for $Z\in \u$. Then, according to the Stone-von
Neumann theorem, there exists a unique, up to equivalence, smooth
irreducible (unitarizable) representation $\SS_{\gamma}$ of $N$ such that $U$
acts by the character $\chi_{\gamma}$. Since $M_{X}$ preserves
$\gamma $ and thus the symplectic form $\kappa_{-1}$, it is well-known \cite{Weil} that the representation $\SS_{{\gamma}}$ extends to the semi-direct product $\cover{M_X}\ltimes N$, where
$\cover{M_X}$ is the inverse image of $M_{X}$ in $\cover{Sp}(\g_{-1})$.
If $\g_{-1}=0$, we will use the same notation $\SS_{{\gamma}}$ to denote the $1$-dimensional
representation of $N=U$ given by the character $\chi_{\gamma}$, and $\cover{M_X}:=M_{X}$ acts on it trivially.

\begin{defn}
\noindent (a)
Let $(\pi,\mathscr{V})$ be a smooth representation of $G$. We define
the space of generalized Whittaker models of $\pi$ associated to the nilpotent orbit $\Orb$
by
\[
\Wh_{\Orb}(\pi)=\Hom_{N}(\mathscr{V},\SS_{{\gamma}}).
\]
$\Wh_{\Orb}(\pi)$ is naturally an $\cover{M_{X}}$-module.

\noindent (b) Let $R$ be a reductive subgroup of $M_{X}$. Also define
\[
\Wh_{\Orb,\tau}(\pi)=\Hom_{\cover{R}N}(\mathscr{V},\mathscr{V}_{\tau}\widehat \otimes
\SS_{{\gamma}}),
\]
where $(\tau,\mathscr{V}_{\tau})$ is a smooth genuine representation of $\cover{R}$.
\end{defn}

\begin{rmk} In the representation theory literature, elements of $\Wh_{\Orb,\tau}(\pi)$ are referred to as (local) models, which are generalizations of Bessel and Fourier-Jacobi models. (Typically one takes $R$ to be a relatively big (e.g. spherical) subgroup of $M_X$.) We refer the reader to \cite{GGP} for more information on these special models such as uniqueness. See also \cite{Ji07} which discusses their roles in a general theory of periods of automorphic forms.
\end{rmk}

We now assume that $G$ is a type I classical group defined by an $\epsilon$-Hermitian space $(V,B)$, where $B$ is the $\epsilon$-Hermitian form.
Set $\g_{\gamma}=\Span_{\rF}\{X,H,Y\}\subset \g$. Under the action of $\g_{\gamma}$, we have the isotypic decomposition:
\begin{equation*}
 V=\bigoplus_{j=1}^{l} V^{\gamma,t_{j}},
\end{equation*}
where $V^{\gamma,t_{j}}$ is a direct sum of irreducible $t_{j}$-dimensional $\g_{\gamma}$-modules, and $t_{1}> t_{2}> \ldots > t_{l}>0$.

We need one more notation. Denote by $(\rho_{m},\rF^{m})$ the irreducible representation of $\g_{\gamma}\simeq \sl_{2}$ of dimension $m$. There is a non-degenerate invariant bilinear form on $\rF^{m}$, which is unique up to scalar and is $(-1)^{m-1}$-symmetric. We fix one such form $(\cdot,\cdot)_{m}$, and denote the corresponded formed space by $(\rF^{m}, (\cdot,\cdot)_{m})$.

Let $V^{\gamma,t_{j}}_{t_j-1}$ be the space of highest weight vectors in $V^{\gamma,t_j}$ (the subscript refers to the $H$-weight), and let $i_j=\dim V^{\gamma,t_j}_{t_j-1}$, which is the multiplicity of $(\rho_{t_j},\rF^{t_j})$ in $V$. The isotypic decomposition gives rise to a partition or equivalently a Young diagram $\mathbf{d}^{\gamma}=[t_{1}^{i_{1}},\ldots,t_{l}^{i_{l}}]$ of size $\dim V$. Using $\sl_{2}$ theory and the $\epsilon$-Hermitian form $B$ on $V$, one may define a non-degenerate $\epsilon_j$-Hermitian form $B^{\gamma,t_{j}}_{t_j-1}$ on $V^{\gamma,t_{j}}_{t_j-1}$, where $\epsilon_j=(-1)^{t_j-1}\epsilon$.
As $\epsilon$-Hermitian spaces, we have
\begin{equation}
\label{eq:adm}
(V,B)\simeq \bigoplus _{j}(V^{\gamma,t_{j}}_{t_{j}-1},B^{\gamma,t_{j}}_{t_{j}-1}) \otimes (\rF^{t_j},(\cdot,\cdot)_{t_j}).
\end{equation}
We may also describe the stabilizer group $M_X$ as follows. Since $M_{X}$ acts on $V^{\gamma,t_{j}}_{t_{j}-1}$ preserving $B^{\gamma,t_{j}}_{t_{j}-1}$, there is an embedding of $G(V^{\gamma,t_{j}}_{t_{j}-1})$ into $M_{X}$, and we identify $G(V^{\gamma,t_{j}}_{t_{j}-1})$ with its image in $M_X$. Then we have
\begin{equation}\label{MX}
M_{X}\cong \prod_{j=1}^{l} G(V^{\gamma,t_{j}}_{t_{j}-1}).
\end{equation}

It is well-known \cite{CM92} that the pair, consisting of
\begin{itemize}
\item[(i)] the Young diagram $\mathbf{d}^{\gamma}=[t_{1}^{i_{1}},\ldots,t_{l}^{i_{l}}]$, and the collection of
\item[(ii)] the formed spaces $(V_{\gamma}(j),B_{\gamma}(j)):= (V^{\gamma,t_{j}}_{t_{j}-1},B^{\gamma,t_{j}}_{t_{j}-1})$ of dimension $t_j$ for $1\leq j\leq l$,
\end{itemize}
uniquely determines the $\sl_{2}$-triple $\gamma$ up to the Adjoint action of $G$. This is an example of an admissible $\epsilon$-Hermitian Young tableaux (or
admissible Young tableaux in short; admissibility refers to the requirement in \eqref{eq:adm}). With an obvious notion of equivalence, nilpotent orbits in $\g$ are then parameterized by equivalence classes of admissible $\epsilon$-Hermitian Young tableaux. See \cite[Section 3.2]{GZ}).


\subsection{Generalized descent of $\sl_{2}$-triples}
\label{subsec:GD}

We are back in the setting of Section \ref{sec:DPreview}, namely we are given $(G,G')=(G(V),G(V'))\subset \Sp(W)$, a type I reductive dual pair.

To ease notation, we will skip $D$ in $\Hom_D (V,V')$ from now on. Recall the moment maps: \cite{KP82,DKPC}
\[
\begin{diagram}
 & &  \Hom (V,V') & &   \\
&  \ldTo^{\varphi} & &\rdTo^{{\varphi}'} &  \\
\g&  & & &{\g}'
\end{diagram}
\]
where $\varphi (T)= T^{\ast}T$ and ${\varphi}'(T)=TT^{\ast}$.

Define
\begin{equation}
\label{eq:Gen}
\Gen\Hom(V,{V}') =\{ T\in \Hom(V,V')\, | \, \mbox{$\Ker (T)$ is non-degenerate}\}.
\end{equation}


For a given $\sl_{2}$-triple $\gamma=\{X,H,Y\}\subset \g$, set $V_{k}=\{v\in V\, | \, Hv=kv\}$, for $k\in \Z$. We have by standard $\sl_{2}$-theory,
\begin{equation*}
V=\bigoplus_{k\in \Z} V_{k}.
\end{equation*}
Similar notations apply  for an $\sl_{2}$-triple $\gamma' =\{X',H',Y'\}\subset \g'$.

\begin{defn}
Let $\gamma\subset \g,$ ${\gamma}'\subset {\g}'$ be two $\sl_{2}$-triples, and $T\in \Hom(V,{V}')$. We say that $T$ lifts $\gamma$ to ${\gamma}'$ if
\begin{enumerate}
\item $T\in \Gen\Hom(V,{V}')$.
\item $\varphi (T) =X$ and $\varphi '(T) ={X}'$.
\item $T(V_{k})\subset {V'_{k+1}}$ for all $k$.
\end{enumerate}
We set
\begin{equation}
\label{eq:Ogg'}
\Orb_{\gamma,{\gamma}'}=\{T\in \Hom(V,{V}')\, | \, \mbox{$T$ lifts $\gamma$ to ${\gamma}'$}\}.
\end{equation}
\end{defn}

\begin{lem}
\label{lem:descent}
Let $\gamma'=\{X',H',Y'\}\subset \g'$ be an $\sl_{2}$-triple. If $X'$ is in the image of $\varphi '$, then there is a unique conjugate class of
$\sl_{2}$-triple $\gamma=\{X,H,Y\}\subset \g$, such that $\Orb_{\gamma,{\gamma}'}$ is non-empty. Furthermore $\Orb_{\gamma,{\gamma}'}$ is a single $M_{X}\times {M'}_{{X'}}$-orbit.
\end{lem}
\begin{proof} We will only show the existence part, namely there is an $\sl_{2}$-triple $\gamma=\{X,H,Y\}\subset \g$ such that $\Orb_{\gamma,{\gamma}'}$ is non-empty.

Decompose \[
 V'=\bigoplus_{j=1}^{l} (V')^{\gamma',t_{j}},
\]
into $\sl_{2}$-isotypic components, where $(V')^{\gamma',t_{j}}$ is a direct sum of irreducible $t_{j}$-dimensional $\sl_{2}$-modules, and $t_{1}> t_{2}> \ldots > t_{l}>0$.
We fix one such irreducible module $\rF ^{r+1}$ of dimension $r+1$, as well as a basis $\{v'_{-r}, v'_{-(r-2)}, ..., v'_{r-2}, v'_{r}\}$ consisting of $H'$-weight vectors.

Write $X'=\varphi '(S)$, for some $S\in \Hom(V,{V}')$. Define the following set of vectors in $V$:
\[v_{-(r-1)}=S^*v'_{-r}, v_{-(r-3)}=S^*v'_{-(r-2)}, ..., v_{r-3}=S^*v'_{r-4}, v_{r-1}=S^*v'_{r-2}.\]
It is easy to see that $\{v_{-(r-1)}, v_{-(r-3)}, ..., v_{r-3}, v_{r-1}\}$ span a non-degenerate subspace of $V$, of dimension $r$.

Now define a new element $T\in \Hom(V,{V}')$ by setting
\[
\begin{aligned}
T^*=\begin{cases} S^*, \ \ \text{ on } v'_{-r}, v'_{-(r-2)}, ..., v'_{r-4}, v'_{r-2}, \\
           0, \ \ \text{ on } v'_{r}. \end{cases}
           \end{aligned}
           \]
Clearly the image of $T^*$ will be non-degenerate, and so is $\Ker (T)$. Putting together the above construction for all $(V')^{\gamma',t_{j}}$'s, the newly defined $T$ and $X:=\varphi (T)$
will have all the required properties.

The rest of proof follows the same line of argument as the proof of Proposition 5.3 and Lemma 5.7 of \cite{GZ}, which is the special case when $\Ker (T)=0$. \end{proof}

\begin{defn} We are in the setting of Lemma \ref{lem:descent}. We will say that $\gamma $ (resp. $\Orb $) is the generalized descent of $\gamma ' $ (resp. $\Orb '$), and we write
\[\Orb = \gDD_{V',V}(\Orb ').\]
When $T\in \Orb_{\gamma,{\gamma}'}$ is injective, we will refer to it as the descent case, and we call $\gamma $ (resp. $\Orb $) the descent of $\gamma ' $ (resp. $\Orb '$). We also write
\[\Orb = \DD_{V',V}(\Orb ').\]
\end{defn}

We will describe the generalized descent $\Orb ' \mapsto \Orb $ using the parametrization of nilpotent orbits by admissible Young tableaux.

Let the nilpotent orbit $\Orb' \subset \g'$ correspond to the admissible $\epsilon'$-Hermitian Young tableaux $(\mathbf{d}^{{\gamma'}},(V'_{\gamma'}(j), B_{\gamma'}(j))$, with
\[\mathbf{d}^{\gamma'}=[(t_{1}+1)^{i_{1}},\ldots,(t_{l}+1)^{i_{l}},2^{a},1^{s}], \quad t_l\geq 2,\]
and $a$ (resp. $s$) is the dimension of the $\epsilon$-Hermitian space $((V')^{\gamma',2}_{1},B^{{\gamma'},2}_{1})$
(resp. $\epsilon'$-Hermitian space $((V')^{\gamma',1}_{0},B^{{\gamma'},1}_{0})$.

If $\Orb'$ is in the image of the moment map, then we have an embedding of $\epsilon$-Hermitian spaces (via $T^*$):
\[
\{\oplus _{1\leq j\leq l}((V')^{\gamma',t_{j}+1}_{t_{j}},B^{{\gamma'},t_{j}+1}_{t_{j}})\otimes (\rF^{t_j},(\cdot,\cdot)_{t_j})\} \oplus ((V')^{\gamma',2}_1, B^{\gamma', 2}_1)\otimes (\rF^{1},(\cdot,\cdot)_{1})\hookrightarrow (V,B).\]
Denote by $(U,B_U)$ be the orthogonal complement of $\{\oplus _{1\leq j\leq l}((V')^{\gamma',t_{j}+1}_{t_{j}},B^{{\gamma'},t_{j}+1}_{t_{j}})\otimes (\rF^{t_j},(\cdot,\cdot)_{t_j})\}$ in $V$ with respect to the form $B$.
Then $(U,B_U)$ contains \[(U_1,B_{U_1}):= ((V')^{\gamma',2}_1, B^{\gamma', 2}_1)\otimes (\rF^{1},(\cdot,\cdot)_{1})\simeq ((V')^{\gamma',2}_1, B^{\gamma', 2}_1)\] as a non-degenerate $\epsilon$-Hermitian subspace.

The admissible $\epsilon$-Hermitian Young tableau $({\mathbf d}^{\gamma},(V_{\gamma}(j),B_{\gamma}(j))$ corresponding to $\Orb\subset \g$ is as follows.
\begin{itemize}
\item ${\mathbf d}^{\gamma}=[t_{1}^{i_{1}},\ldots,t_{l}^{i_{l}}, 1^{a+b}]$, where $\dim U=a+b$.
\item $(V^{\gamma,t_{j}}_{t_{j}-1},B^{\gamma,t_{j}}_{t_{j}-1})\cong ((V')^{\gamma',t_{j}+1}_{t_{j}},B^{{\gamma'},t_{j}+1}_{t_{j}})$, for all $j=1,\ldots,l$;
\end{itemize}
Or in descriptive words:
\begin{itemize}
\item The Young diagram of $\Orb$ is obtained by erasing the first column off the Young diagram
of $\Orb '$, and then adding some additional boxes (this is the $b$) in the first column.
\item After erasing the first column off the Young diagram of $\Orb '$, one keeps the forms on the spaces of multiplicities unchanged.
\end{itemize}

In the setting of the generalized descent, the two stabilizer groups $M_X$ and $M'_{X'}$ are closely related. We proceed to describe their relationship.

\begin{defn} Let $\gamma$, ${\gamma}'$ be as before, and $T\in \Orb_{\gamma,{\gamma}'}$. Define the following non-degenerate subspaces of $V$ and $V'$:
\[
V_{\gamma,{\gamma}'}:=\Ker (T) \qquad \text{and } \ V'_{\gamma, \gamma'}:=(V')^{\gamma',1}_{0} \ \ (\text{the space of $\gamma'$-invariants}).
\]
Let $L$ and $L'$ be the isometry groups of $V_{\gamma,{\gamma}'}$ and $V'_{\gamma, \gamma'}$, respectively.
\end{defn}

We define a group homomorphism
\begin{equation}
\label{eq:alp}
\alpha=\alpha_{T}:{M'}_{X'}\longrightarrow M_{X}
\end{equation}
by setting $\alpha _{T}(m')$ to be the identity map on $\Ker (T)$, and equal to $T^{-1}m'T$ on $(\Ker T)^{\perp}$, for $m'\in {M'}_{X'}$.

We have a direct product
\begin{equation}
\label{eq:stab'}
M'_{X'}=M_{X,X'}\times L', \quad \text{where} \quad
M_{X,X'}=\prod_{j=1}^{l} G((V')^{\gamma',t_{j}+1}_{t_{j}}) \times G(U_1).
\end{equation}
Similarly we have
\begin{equation}
\label{eq:stab}
M_{X} = \prod_{j=1}^{l} G(V^{\gamma,t_{j}}_{t_{j}-1}) \times G(U)\cong \prod_{j=1}^{l}G((V')^{\gamma',t_{j}+1}_{t_{j}}) \times G(U).
\end{equation}
Note that the latter contains $M_{X, X'}\times L$ as a symmetric subgroup, due to the fact that $U=U_1\oplus \Ker (T)$ (orthogonal direct sum) and so $G(U)\supset G(U_1)\times L$ as a symmetric subgroup. Under the above identification, $\alpha : {M'}_{X'}\rightarrow M_{X}$ is the identity homomorphism on the first factor of \eqref{eq:stab'}, and the trivial homomorphism on the second factor of  $\eqref{eq:stab'}$.

\subsection{Transition of generalized Whittaker models}
\label{subsec:Trans}

We are in the setting of Section \ref{sec:DPreview}: we are given $(G,G')\subset \Sp(W)$, a type I reductive dual pair, and the smooth oscillator representation $(\omega, \SY)$ of $\cover{\Sp}(W)$ associated to a fixed $\psi$. We are interested in knowing all generalized Whittaker models of $\Theta (\pi)$, where $\pi \in \Irr(\cover{G})$.

We state our result on the transition of generalized Whittaker models. This is an extension of the main result of \cite{GZ}, which is the descent case (i.e., $\Ker (T)=0$, in which case $L$ is trivial and $M_X=M_{X,X'}$) in the terminology of this article.

\begin{thm} {\em {([Gomez and Zhu])}}
\label{thm:whittaker}
Let $\pi \in \Irr (\cover{G})$.

\noindent (a) If $\Orb'$ is not in the image of the moment map $\varphi'$, then
\[\Wh_{\Orb'}(\Theta (\pi))=0.\]

\noindent (b) If $\Orb'$ is in the image of the moment map $\varphi'$, let $\Orb = \gDD_{V',V}(\Orb ')$ be its generalized descent.  We have (as in \eqref{eq:stab'} and \eqref{eq:stab})
\[
\begin{aligned}
M_{X}\supset M_{{X}, {X}'}\times L, \\
M'_{X'}=M_{{X}, {X}'}\times L', \\
\end{aligned}
\]
and $(L,L')$ forms a dual pair of the same type as $(G,G')$. Then for any $\tau '\in \Irr(\cover{L'})$, we have
\[\Wh_{\Orb',\tau '}(\Theta (\pi))\simeq \Wh_{\Orb, \Theta (\tau')^{\vee}}(\pi ^{\vee})\]
as $\cover{M_{X, X'}}$-modules.  Here $\Theta (\tau')$ is the full theta lift of $\tau '$ with respect to the dual pair $(L,L')$, and the symbol $\vee$ indicates the dual in the category of Casselman-Wallach representations.
\end{thm}

Let $\gamma\subset \g$ and ${\gamma}'\subset {\g}'$ be two $\sl_{2}$-triples of type $\Orb$ and $\Orb '$, where $\Orb = \gDD_{V',V}(\Orb ')$. We may assume that
\begin{equation}
\label{eq:VV'}
V=\bigoplus_{k=-r}^{r} V_{k} \qquad \mbox{and} \qquad {V'}=\bigoplus_{k=-r-1}^{r+1} {V'}_{k},
\end{equation}
for some $r$. Write $T=\oplus _{k=-r}^{r}T_k$, where $T_k=T|_{V_k}$. The following diagram is instructive:
\begin{equation}
\label{eq:figT}
\begin{array}{c}   V_{-r}   \oplus    V_{-r+1}    \oplus    \cdots    \oplus    V_{r-1}   \oplus   V_{r}   \\
  \hspace{45pt} \searrow \hspace{-3pt} {\scriptstyle T_{-r}}  \hspace{10pt}  \searrow \hspace{-3pt} {\scriptstyle T_{-r+1}} \hspace{5pt}  \cdots\hspace{5pt}    \searrow \hspace{-3pt} {\scriptstyle T_{r-1}} \hspace{-5pt}   \searrow \hspace{-3pt} {\scriptstyle T_{r}}    \\
{V'}_{-r-1}    \oplus     {V'}_{-r}     \oplus     {V'}_{-r+1}    \oplus    \cdots    \oplus    {V'}_{r-1}   \oplus   {V'}_{r}    \oplus   {V'}_{r+1}.
\end{array}
\end{equation}
We have $\Ker (T)=\Ker (T_0)\subseteq V_0$, and $T_k$ is injective for $k\ne 0$, by our non-degeneracy assumption on $\Ker (T)$.

Recall that $(G,G')\subseteq \Sp (W)$ is a type I reductive dual pair, and $(\omega , \SY )$ is the smooth oscillator representation of $\cover{\Sp}(W)$ associated to a non-trivial unitary character $\psi$ of $\rF$, which we fix. Denote by $\SY_{{U'},\chi_{{\gamma}'}}$ the space of $({U'},\chi_{{\gamma}'})$-covariants of the smooth oscillator representation $(\omega , \SY )$.

The key technical result to derive Theorem \ref{thm:whittaker} is the following

\begin{prop}\label{prop:mainproposition}
 Let $\gamma\subset \g$ and ${\gamma}'\subset {\g}'$ be two $\sl_{2}$-triples of type $\Orb$ and $\Orb '$, where $\Orb = \gDD_{V',V}(\Orb ')$. Then, given any $T\in \Orb_{\gamma,{\gamma}'}$, there exists a $ \cover{G}\times \cover{M'_{X'}}{N'}$-intertwining isomorphism
\begin{equation*}
\Psi_{T}:\SY_{{U'},\chi_{{\gamma}'}}\longrightarrow \Sch(\cover{L}N\backslash \cover{G};\SS_{\check{\gamma}}\otimes \SS_{{\gamma}'} \otimes \SY_{\gamma,{\gamma}'}),
\end{equation*}
where $\check{\gamma} =\{-X,H,-Y\}$, and $\SY_{\gamma,{\gamma}'}$ is the smooth oscillator representation associated to the symplectic subspace $W_{\gamma,\gamma'}:=\Hom (V_{\gamma, \gamma '},
V'_{\gamma, \gamma '})$ of $W$.
 \end{prop}

\begin{rmk} In the above proposition and in the sequel, the unexplained notation $\Sch$ denotes certain space of rapidly decreasing functions defined by an appropriate collection of semi-norms. See \cite[Section 4.3]{GZ}. We also refer to \cite[Proposition 6.3]{GZ} for other unexplained actions such as the action of $\cover{M'_{X'}}{N'}$ on the right hand of the isomorphism $\Psi_{T}$.
\end{rmk}

In the descent case, $V_{\gamma,\gamma'}=0$, $L$ is the trivial group, $M_X=M_{X,X'}$, and the above proposition is just \cite[Proposition 6.5]{GZ}. The proof in the general case is along the same line, and it is by induction on $r$. We will outline the main ingredients of the inductive step in the next subsection.

\subsection{The inductive step}
To facilitate the discussion, we introduce some notations. For $l\leq r$ and $m\leq r+1$, set
\[
V_{(l)}=\bigoplus_{k=-l}^{l}V_{k}, \qquad \mbox{and} \qquad {V'}_{(m)}=\bigoplus_{k=-m}^{m}{V'}_{k}.
\]
Clearly both $V_{(l)}$ and ${V'}_{(m)}$ are non-degenerate. Let $G_{(l)}:=G(V_{(l)})$ (the isometry group of $V_{(l)}$), which we view as a subgroup of $G$ in the obvious way.
We define ${G'}_{(m)}$ similarly.
Let $(\omega_{(l),(m)},\SY_{(l),(m)})$ be the smooth oscillator
representation associated to the dual pair
$(G_{(l)},{G'}_{(m)})$ and the character $\psi$ of $\rF$. Let the symbol $\SS$ stand for the Schwartz space (of a vector space). Set
$\SS_{(l),m}=\SS(\Hom(V_{(l)},{V'}_{m}))$, $\SS_{l,(m)}=\SS(\Hom(V_{l},{V'}_{(m)}))$ and
$\SS_{l,m}=\SS(\Hom(V_{l},{V'}_{m}))$. Note that $\SY_{(l),(m)}$ can be identified with the
Schwartz space of a Lagrangian subspace of $\Hom(V_{(l)},{V'}_{(m)})$, through the Schrodinger model.

Under these notations, we have
\[V=V_{(r)}, \qquad \mbox{and} \qquad V'={V'}_{(r+1)};\]
\[G=G_{(r)}, \qquad \mbox{and} \qquad G'=G'_{(r+1)};\]
\[(\omega ,\SY)=(\omega_{(r),(r+1)},\SY_{(r),(r+1)}).\]

Let ${P'}_{m}$ be the stabilizer of ${V'}_{m}$ in ${G'}_{(m)}$. We have ${P'}_{m}={M'}_{(m)}{N'}_{m}$, where ${N'}_{m}$ is the
unipotent radical of ${P'}_{m}$, ${M'}_{(m)}={M'}_{m}\times {G'}_{(m-1)}$, and ${M'}_{m}\cong \GL ({V'}_{m})$.
Set
\begin{equation*}
N'_{(m)}=N'\cap G_{(m)}, \qquad U'_{m}=U'\cap N'_{m}, \qquad \mbox{and} \qquad U'_{(m)}=U'\cap N'_{(m)}.
\end{equation*}

Recall that $\chi'=:\chi_{\gamma'}$ is the character of $U'$ associated to the $\sl_{2}$-triple $\gamma'$, as in \eqref{defchi}. We set
\[\chi'_{m}=\chi'|_{U'_{m}}, \ \ \text{ and } \ \ \chi'_{(m)}=\chi'|_{U_{(m)}}.\]
Using these notations, we then have
\begin{equation}
\label{GandC}
U'=U'_{(r+1)}=U'_{(r)}U'_{r+1}, \ \ \text{ and } \ \ \chi '=\chi '_{(r+1)}=\chi '_{(r)}\chi '_{r+1}.
\end{equation}

Instead of working with the covariant space $\SY_{{U'},\chi' }$, it is more convenient to work dually with the space of quasi-invariant distributions $(\SY')^{{U'},{\chi'}}$.  It is also convenient to adopt the notion of $\SV$-distributions on a manifold, where $\SV$ is a Fr\'echet space. We refer the reader to \cite{KV96} for a general introduction on such distributions.

To carry out the induction, we will apply the decomposition
\begin{equation}
\label{decom1}
{V'}_{(r+1)} = {V'}_{r+1}\oplus {V'}_{-r-1}\oplus {V'}_{(r)},
\end{equation}
followed by the decomposition
\begin{equation}
\label{decom2}
V_{(r)}=V_{-r}\oplus V_{r}\oplus V_{(r-1)}.
\end{equation}

This gives rise to the decomposition
\begin{eqnarray*}
&  &\Hom(V_{(r)},{V'}_{(r+1)}) = \Hom(V_{(r)},{V'}_{r+1})\oplus \Hom(V_{(r)},{V'}_{-r-1})\oplus \Hom(V_{(r)},{V'}_{(r)})\\
&  & = (\Hom(V_{(r)},{V'}_{r+1})\oplus \Hom(V_{-r},{V'}_{(r)}))\oplus (\Hom(V_{(r)},{V'}_{-r-1}) \oplus \Hom(V_{r},{V'}_{(r)}))\\
&  & \ \ \ {} {}\oplus \Hom(V_{(r-1)},{V'}_{(r)}).
\end{eqnarray*}
Note that $\Hom(V_{(r)},{V'}_{r+1})\oplus \Hom(V_{-r},{V'}_{(r)})$, $\Hom(V_{(r)},{V'}_{-r-1}) \oplus \Hom(V_{r},{V'}_{(r)})$ are totally isotropic, complementary subspaces, and thus via the mixed models, we will have the identification
\begin{equation}
\label{eq:polarization}
\begin{aligned}
\SY_{(r),(r+1)}&\cong \SS_{(r),r+1}\otimes \SY_{(r),(r)}\\
&\cong [\SS_{(r),r+1}\otimes \SS_{-r,(r)}]\otimes \SY_{(r-1),(r)}, \ \ \ \ r>0.
\end{aligned}
\end{equation}
This identification allows us to do the induction with respect to $r$.

There are two steps in each induction, corresponding to the two isomorphisms in \eqref{eq:polarization}.

\vsp
Our goal is to understand the space $(\SY_{(r),(r+1)}')^{{U'}_{(r+1)},{\chi'}_{(r+1)}}$. In view of \eqref{GandC}, we start with a (tempered) distribution $\lambda \in (\SY_{(r),(r+1)}')^{{U'}_{r+1},{\chi'}_{r+1}}$.

\noindent {\bf Step 1}: We use the explicit realization of $\SY_{(r),(r+1)}$ given by the first isomorphism in \eqref{eq:polarization}. We identify $\lambda$ with a
$\SY_{(r),(r)}'$-valued distribution on $\Hom(V_{(r)},{V'}_{r+1})$, as in \cite{KV96}.
By using (derived) action of elements in the center of $\u'_{r+1}$ and the fact that $\Im (X')$ contains $V'_{r+1}$, one shows that $\lambda$ may be identified with a $\SY_{(r),(r)}'$-valued distribution that \emph{lives on} $\Hom_{\GNM}(V_{(r)},{V'}_{r+1})$ (i.e., having transverse order zero at all points of $\Hom_{\GNM}(V_{(r)},{V'}_{r+1})$), where
\begin{equation*}
\begin{aligned}
&\Hom_{\GNM}(V_{(r)},{V'}_{r+1})= \\
&\{T\in \Hom(V_{(r)},{V'}_{r+1})\, | \, \mbox{$TT^{\ast}=0$ and $T$
has maximal rank $\dim {V'}_{r+1}$} \}.
\end{aligned}
\end{equation*}
(Here the subscript $\GNM$ stands for generic null mappings.)

\begin{equation*}
\text{Assume that $\Hom_{\GNM}(V_{(r)},{V'}_{r+1})$ is non-empty.}\tag{A1}
\end{equation*}
Then $\Hom_{\GNM}(V_{(r)},{V'}_{r+1})$ is a single orbit under the action of $G=G_{(r)}$. We pick one representative of $\Hom_{\GNM}(V_{(r)},{V'}_{r+1})$
called $T_r$, and by permuting under the $G$-action if necessary, we may pick a $T_r$ with the following property: there exists a totally isotropic subspace $V_{r}$ of $V=V(r)$ having the same dimension as ${V'}_{r+1}$, and
 \begin{equation}
 \label{eq:Tr}
 T_r|_{V_{(r-1)}\oplus V_{-r}}=0, \ \text{ and } \ T_r: V_r\rightarrow V'_{r+1} \, \text{is a linear isomorphism}.
 \end{equation}
 Here $V_{-r}$ is the isotropic subspace dual to $V_{r}$, and $V_{(r-1)} = (V_r\oplus V_{-r})^{\perp}$.

\begin{rmk} Note that in Part (a) of Theorem \ref{thm:whittaker}, we are only given the nilpotent orbit $\Orb'$, and not $\Orb$. Thus we have the decomposition \eqref{decom1}, but not the decomposition \eqref{decom2}. Nevertheless, the just concluded discussion allows us, under the hypothesis (A1), to find $V_r$, $V_{-r}$, $V_{(r-1)}$ so that the decomposition \eqref{decom2} holds. Of course these spaces ($V_r$, $V_{-r}$, $V_{(r-1)}$) no longer carry any meanings in terms of eigenvalues of the neutral element of an $\sl_2$-triple. With the decomposition \eqref{decom2} in place, we then continue the induction.
\end{rmk}

Let $P_{r}$ be the stabilizer of $V_{-r}$
in $G$. Then $P_{r}=M_{(r)}N_{r}$, where $N_{r}$ is the
unipotent radical of $P_{r}$, $M_{(r)}= M_{r}\times G_{(r-1)}$, and $M_{r}\cong \GL(V_{-r})$. Note that $N_{r}G_{(r-1)}$ is the stabilizer of $T_r$.

{\bf Step 2}: Since we have the decomposition $V=V_{-r}\oplus V_{(r-1)}\oplus V_r$, we could now use the explicit realization of $\SY_{(r),(r+1)}$ given by the second isomorphism in \eqref{eq:polarization}. After evaluating at $T_r$, and by using (derived) action of other elements of $\u'_{r+1}$, one finds that $\lambda$ may be identified with a distribution
that \emph{lives on} $G\times \SA$, where $\SA$ is the affine space given by the linear system:
\[
\{S \in \Hom(V_{-r},{V'}_{(r)})\, | \,
\Tr ({R'}T_{r}S^{\ast})=\Tr({R'}{X'}),\, \forall \, {R'}\in \Hom({V'}_{r+1},{V'}_{(r-1)}\oplus
{V'}_{-r})\}.
\]

\begin{rmk} The above affine space $\SA$ appears incorrectly in line 18, page 842 of \cite{GZ}: $\Tr$ should be inserted on both sides.
\end{rmk}

\begin{equation*}
\text{Assume that $\SA$ is non-empty}. \tag{A2}
\end{equation*}
Observe that the corresponding homogeneous system of $\SA$ has $\Hom(V_{-r},{V'}_{-r})$ as the general solution.  Due to the way we have arranged $T_r$, we may choose a particular solution $T_{-r}$ or equivalently $T_{-r}^{\ast}$  so that
\begin{equation}
\label{eq:T-r}
T_{-r}^{\ast}\in \Hom (V'_{r-1}, V_r) \text { and } T_r T_{-r}^{\ast}=X' \text { on } V'_{r-1}.
\end{equation}
Figuratively, we have completed a triangle:
\[
\begin{array}{c}   V_{r}   \\
\hspace{-1pt} {\scriptstyle T^{\ast}_{-r}} \nearrow  \searrow \hspace{1pt} {\scriptstyle T_{r}}\\
{V'}_{r-1}   \xrightarrow{X'} V'_{r+1}
\end{array}
\]
The full solution of the linear system is then $T_{-r}+\Hom(V_{-r},{V'}_{-r})$, and $\lambda$ is identified as distribution that \emph{lives on} $G\times [T_{-r}+\Hom(V_{-r},{V'}_{-r})]$.

Given a function $f\in \SY_{(r),(r+1)}\cong [\SS_{(r),r+1}\otimes
\SS_{-r,(r)}]\otimes \SY_{(r-1),(r)}$, define a new function $f_{r}\in
C^{\infty}(G;\SS_{-r,-r}\otimes \SY_{(r-1),(r)})$ by
\[
f_{r}(g)(S)=[\omega_{(r),(r+1)}(g)f](T_{r},T_{-r}+S),\qquad g\in G, S\in\Hom(V_{-r},{V'}_{-r}).
\]

The final result of the inductive step, which will relate covariants of $\SY_{(r),(r+1)}$ with those of $\SY_{(r-1),(r)}$, is the following

\begin{lem}
\label{lem:ind}
\noindent (a) For $(\SY_{(r),(r+1)})_{{U'}_{r+1},{\chi'}_{r+1}}$ to be non-zero, both (A1) and (A2) must be satisfied.

 \noindent (b) Assume both (A1) and (A2). For $r>0$, the map $f\mapsto f_{r}$ induces a $G$-intertwining isomorphism
\[
\Psi_{r}:(\SY_{(r),(r+1)})_{{U'}_{r+1},{\chi'}_{r+1}} \longrightarrow
\Sch(N_{r} G_{(r-1)}\backslash G;\SS_{-r,-r}\otimes\SY_{(r-1),(r)}),
\]
where the action of $N_{r} G_{(r-1)}$ on $\SS_{-r,-r}\otimes\SY_{(r-1),(r)}$ is given by equations (6.37)-(6.39) in \cite{GZ}.
\end{lem}

\subsection{Sketch of proof of Theorem \ref{thm:whittaker}} Continuing the induction provided by Lemma \ref{lem:ind}, Part (a) of Theorem \ref{thm:whittaker} follows. This is because the nonvanishing of $(\SY_{(r),(r+1)})_{{U'}_{(r+1)},{\chi'}_{(r+1)}}$ requires both (A1) and (A2) to be satisfied for all $r$. We may thus find $T_r$ and $T_{-r}$ as in equations \eqref{eq:Tr} and \eqref{eq:T-r}.
It is then clear that the moment map $\varphi '$ will send $T=:\oplus _{k=-r}^{r}T_k$ to $X'$, and so $\Orb'$ is in the image of the moment map $\varphi'$.

For Part (b) of Theorem \ref{thm:whittaker}, we proceed with the induction as in the descent case. For details, see Sections 6.4 and 6.5 of \cite{GZ}. At the end of the induction, we will need the following lemma, which is a replacement of \cite[Lemma 5.8]{GZ} in the descent case. Together with the inductive steps, it implies Proposition \ref{prop:mainproposition}, which in turn implies Part (b) of Theorem \ref{thm:whittaker}.

Recall $\g_{-1}$ is a symplectic space whose symplectic structure is defined through an
$\Ad\,G$-invariant non-degenerate $\rF$-bilinear form $\kappa$ on $\g$, as in \eqref{defsymg-1}.
We now fix the bilinear form $\kappa$ as $\frac{1}{2}\Tr(T^{\ast}S)$ (for $T, S \in \g$), and likewise for ${\kappa'}$.

\begin{lem} Let ${W}_{0}$ be the symplectic subspace of $W$ defined by
\begin{equation*}
{W}_{0}=\bigoplus_{k=-r}^{r}
\Hom(V_{k},{V}'_{k})\subset \Hom(V,{V}').
\end{equation*}
 Given $T\in \Orb_{\gamma,{\gamma'}}$, the map
\[
\begin{array}{rcl}
J_{T}: \ \ -\g_{-1}\oplus {\g'}_{-1} & \longrightarrow &  {W}_{0}, \\
(R,{R'})  & \mapsto &  TR+{R'}T
\end{array}
\]
is a symplectic imbedding, with $(\Im J_{T})^{\perp}=W_{\gamma,{\gamma}'}$.
\end{lem}

\begin{rmk} The smooth oscillator representation of $\cover{\Sp} (W_0)$ is realized in the tensor product space $\SS_{-r,-r}\otimes \SS_{-r+1,-r+1}\cdots\otimes \SS_{-1,-1}\otimes \SY_{(0),(0)}$.
This representation naturally appears when we carry out the induction. See Lemma \ref{lem:ind}.
\end{rmk}

\begin{proof}
A straightforward computation shows that
 \[
\langle J_{T}(R_{1},{R'}_{1}),J_{T}(R_{2},{R'}_{2})\rangle=-\kappa_{-1}(R_{1},R_{2})+{\kappa'}_{-1}({R'}_{1},{R'}_{2}),
\]
for all $R_{1}$, $R_{2}\in \g_{-1}$, ${R'}_{1}$, ${R'}_{2}\in {\g'}_{-1}$. Since $\kappa_{-1}$ and ${\kappa'}_{-1}$ are non-degenerate, we see that $J_{T}$ is injective.

It is also straightforward to show that
\[
\langle S,J_{T}(R,R')\rangle = 0, \qquad \mbox{for $S\in W_{0}$, $R\in \g_{-1}$ and ${R'}\in {\g'}_{-1}$}.
\]

It remains to show that the dimensions match. Observe that there is a linear isomorphism between $\oplus_{k\leq 0} \Hom(V_{k},V_{k-1})$ and $\g_{-1}$. From this we conclude that
\begin{eqnarray*}
\dim \g_{-1} & = & \sum_{k=0}^{-r+1} \dim V_{k}\cdot \dim V_{k-1} \\
                     & = & \sum_{k=0}^{-r+1} \dim {V'}_{k-1}\cdot \dim V_{k-1} +\dim (\Ker T) \cdot\dim V_{-1}\\
                     & = & \sum_{k=-1}^{-r} \dim {V'}_{k}\cdot \dim V_{k} +\dim (\Ker T)\cdot\dim V_{-1}.
\end{eqnarray*}
Here we have used the easily checked fact that $\dim V_{0} =\dim {V'}_{-1} +\dim (\Ker T)$ and $\dim V_{k} =\dim {V'}_{k-1}$ for $k\leq 1$. C.f. the diagram in \eqref{eq:figT}.
Similarly,
\begin{eqnarray*}
\dim {\g'}_{-1}& = & \sum_{k=0}^{-r} \dim {V'}_{k}\cdot \dim {V'}_{k-1} \\
                    & = & \sum_{k=0}^{-r} \dim {V'}_{k}\cdot \dim V_{k} -\dim {V'}_{0}\cdot\dim (\Ker T) \\
                    & = & \sum_{k=0}^{r} \dim {V'}_{k}\cdot \dim V_{k} -\dim {V'}_{0}\cdot \dim (\Ker T).
\end{eqnarray*}
Here we have the fact that $\dim V_{k}=\dim V_{-k}$. Therefore we have
\begin{eqnarray*}
\dim \g_{-1}+\dim {\g'}_{-1} & = & \sum_{k=-r}^{r}\dim {V'}_{k}\cdot \dim V_{k}-\dim (\Ker T) \cdot[\dim {V'}_{0}-\dim V_{-1}]\\
                              & = & \dim ({W}_{0}) -\dim (\Ker T)\cdot[\dim V'_{0}-\dim V'_{-2}]\\
& = & \dim ({W}_{0}) -\dim (\Ker T)\cdot\dim ((V')^{{\gamma '},1}_0).
\end{eqnarray*}
\end{proof}

\section{Correspondence of associated characters}
\label{sec:AC}

In this section, the local field $\rF$ will be $\R$. We will discuss a recent result of Ma, Sun and the author \cite{MSZ}, which computes the associated character of the local theta lift, in the so-called convergent range. Related earlier works on transition of associated cycles include \cite{NOTYK,NZ,LM}.

\subsection{The associated character map}
\label{subsec:AC}
We first recall some basics of Vogan's theory of associated varieties \cite{Vo89}.  We will need a minor extension of the theory to allow covering groups which appear in the theory of dual pairs.

Let $\bfG$ be a complex reductive Lie group and $\g$ is its Lie algebra. Denote by $\Nil_{\bfG}(\g)$ the set of nilpotent $\bfG$-orbits in
$\g$. We will be concerned with genuine Casselman-Wallach representations of $\cover{G}$ and their Harish-Chandra modules. Here $(G,G')\subseteq \Sp (W)$ is a type I reductive dual pair as in Section \ref{sec:DPreview}. In this context, $\bfG$ is the complexification of $G$, which is a complex classical group.

Suppose $\Orb \in \Nil_{\bfG}(\g)$. We say that a finite length $(\g,\cover{K})$-module $\Pi$ is \emph{$\Orb$-bounded} (or \emph{bounded by $\Orb$}) if the associated variety  of the annihilator ideal $\Ann(\Pi)$ is contained in $\bcO$ (the closure of $\Orb$). Let
\[
  \g=\k\oplus \p
\]
be the complexified Cartan decomposition fixed by our choice of the maximal compact subgroup $K$ of $G$.
It follows from \cite[Theorem 8.4]{Vo89}
that  $\Pi$ is $\CO$-bounded if and only if its associated variety $\AV(\Pi)$ is contained in
$\overline \Orb \cap \p$.
Let
$\CM_{\Orb}(\g,\cover{K})$ denote the category of genuine
$\Orb$-bounded
finite length $(\g,\cover{K})$-modules, and write
$\CK_{\Orb}(\g,\cover{K})$ for its Grothendieck group.

Under the adjoint action of $\bfK$, the complex variety $\CO\cap \p$ is a union of finitely many
orbits, each of dimension $\frac{\dim_\C \CO}{2}$.  For any $\bfK$-orbit
$\sO\subset \CO\cap \p$, let $\CK_{\sO}(\wt{\bfK})$ denote the Grothendieck group of $\wt{\bfK}$-equivariant coherent sheaves (or the Grothedieck group of the category of $\wt{\bfK}$-equivariant algebraic vector bundles) on $\sO$.
Taking the isotropy representation at a point $X\in \sO$ yields an identification
\begin{equation}\label{idenkr}
  \CK_{\sO}(\wt{\bfK})=\CR(\wt{\bfK}_{X}),
\end{equation}
where the right hand side denotes  the Grothendieck group  of the category of algebraic representations of the stabilizer group $\wt{\bfK}_X$.

Put
\begin{equation}
\label{eq:dec.KO}
  \CK_{\CO}(\wtbfK):=\bigoplus_{\sO\textrm{ is a $\bfK$-orbit in $\CO\cap \p$}}
  \mathrm \CK_{\sO}(\wt{\bfK}).
\end{equation}

According to Vogan \cite[Theorem~2.13]{Vo89},  we have a canonical homomorphism, called the associated character map:
\begin{equation}
\label{def:Ch}
  \Ch_{\Orb }: \CK_{\Orb}(\g, \cover{K})\rightarrow \CK_{\CO}(\wtbfK).
\end{equation}
Briefly, starting from a good filtration of a finite length $(\g,\cover{K})$-module $\Pi$, one obtains a finitely generated $(S(\g), \cover{K})$-module $M$ by taking the graded module. Vogan shows that there is a finite filtration $\{M_j\}$ of $M$ by $(S(\g), \cover{K})$-submodules with the property that every subquotient is generically reduced along every minimal prime in the associated variety of $M$. The virtual representation $\Ch_{X}$ of $\wt{\bfK}_X$ that Vogan attaches to $X\in \sO$ is then
\[\Ch_{X}=\sum_{j}M_j/(\m(X)M_j+M_{j-1}),\]
where $\m(X)$ is the maximal ideal in $S(\g)$ corresponding to $X$. It is independent of all the choices that have been made. See \cite[Section 2]{Vo89} for details.

For a Casselman-Wallach representation of $\cover{G}$, we define its associated character by using its Harish-Chandra module.

\subsection{Algebraic theta lifting}\label{sec:ATL}
We are back in the setting of Section 1, namely we are given $(G,G')=(G(V),G(V'))\subset \Sp(W)$, a type I reductive dual pair, and the smooth oscillator representation $(\omega, \SY)$ of $\cover{\Sp}(W)$ associated to a fixed $\psi$.
Let $\Omega _{V,V'}$ be the Harish-Chandra module of $(\omega, \SY)$, which is
 naturally  a $(\g\times \g', \cover{K}\times \cover{K'})$-module.

\begin{defn}
For a genuine $(\g,\cover{K})$-module $\Pi$ of finite length,
define
\[
  \Thetav_{V,V'}(\Pi ):= \left(\Omega _{V,V'}\otimes \Pi \right)_{\g, \wtK},\qquad
  \text{(the coinvariant space).}
\]
The $(\g',\wtK')$-module $\Thetav_{V,V'}(\Pi)$, or $\Thetav(\Pi)$ in short,
is genuine and of finite length \cite{Ho89}.
\end{defn}

\begin{rmk} If $\pi$ is a smooth irreducible genuine representation of
$\cover{G}$, and let $\Pi$ denote the Harish-Chandra module of $\pi $. Then
\[
  \Thetav (\Pi )\simeq \Theta (\pi ^{\vee})^{\text{al}},\]
  where $\pi ^{\vee}$ is the contragredient representation of $\pi$, and $\Theta (\pi ^{\vee})$ is the full theta lift of
  $\pi ^{\vee}$ (in the smooth category), and $\Theta (\pi ^{\vee})^{\text{al}}$ is the Harish-Chandra module of $\Theta (\pi ^{\vee})$.
  \end{rmk}

For the moment, we let $(\bfG,\bfG')\subseteq \Sp(\bfW)$ be a complex reductive dual pair \cite{Ho79}. Recall we have the moment maps:
\cite{KP82,DKPC}
\[
\begin{diagram}
 & &  \bfW & &   \\
&  \ldTo^{\MM} & &\rdTo^{\MM'} &  \\
\g&  & & &{\g}'
\end{diagram}
\]

Let us describe the moment maps explicitly for a type I complex dual pair.
Let $\bfV$ be an $\epsilon $-symmetric  bilinear space and $\bfV'$ be an $\epsilon '$-symmetric bilinear space, where
  $\epsilon ,\epsilon '\in \{\pm1\}$ with $\epsilon \epsilon '=-1$. We have the complex symplectic space
  $\bfW := \Hom(\bfV,\bfV')$, with the symplectic form:
       \[
    <T_1, T_2>_{\bfW} := \Tr(T_1^\star T_2), \qquad T_1,T_2\in \bfW.
  \]
  Here $\star\colon \Hom(\bfV,\bfV')\rightarrow \Hom(\bfV',\bfV) $ is the adjoint map induced by the form $(\cdot , \cdot  )_{\bfV}$ on $\bfV$ and the form $(\cdot , \cdot  )_{\bfV'}$ on $\bfV'$:
  \[
    <Tv, v'>_{\bfV'} = <v,T^\star v'>_{\bfV},  \qquad v\in
    \bfV,v'\in \bfV', T\in \Hom(\bfV,\bfV').
  \]
  Then $(\bfG,\bfG'):=(\bfG_{\bfV},\bfG_{\bfV'})\subseteq \Sp(\bfW)$. The moment maps are then given by
  \[\MM (T)= T^{\star}T, \ \ \text{ and } \ \ \MM '(T)=TT^{\star}.\]

We also recall the notion of theta lift of nilpotent orbits for a complex dual pair $(\bfG,\bfG')\subseteq \Sp(\bfW)$. By \cite[Theorem 1.1]{DKPC}, for any $\CO\in \Nil_{\bfG}(\g)$, $\MM'(\MM^{-1}(\bcO))$ equals the closure of a
 unique nilpotent orbit $\CO' \in \Nil_{\bfG'}(\g')$. We call $\CO'$ the \emph{theta lift} of $\CO$, and we write
 \begin{equation}
 \label{def:LC}
  \CO'=\oliftc_{\bfV,\bfV'}(\CO).
 \end{equation}

Now we fix compatible Cartan forms $L$ on $\bfV$, $L'$ on $\bfV'$ and $L_{\bfW}$ on $\bfW$ (as in \cite[Section 2]{MSZ}), and this leads us to
a real reductive dual pair in $(G,G')\subseteq \Sp(W)$ (also called a rational dual pair). Denote $\bfK =\bfG^{L}$ and $\bfK'=(\bfG')^{L'}$.

Let $\bfii$ denote a fixed $\sqrt{-1}$. We decompose
\[
\bfW = \CX\oplus \CY
\]
where $\CX$ and $\CY$ are $+\bfii$ and $-\bfii$ eigenspaces of $L_{\bfW}$, respectively. Restriction on $\CX$ induces a pair of maps:
\[
\begin{diagram}
 & &  \CX & &   \\
&  \ldTo^{\MMP:=\MM|_{\CX}} & &\rdTo^{\MMP':=\MM'|_{\CX}} &  \\
\p&  & & &{\p}'
\end{diagram}
\]
which are also called moment maps (with respect to the rational dual pair).

We have the following estimate of the size of
$\Thetav(\pi)$.

\begin{prop}[{\cite[Theorem~B and Corollary~E]{LM}}]\label{cor:Cbound}
For any genuine $(\g,\wtK)$-module $\Pi$ of finite length, we have
\[
\AV(\Thetav(\Pi)) \subset M'(M^{-1}(\AV(\Pi))).
\]
Consequently, if $\Pi$ is $\CO$-bounded for a nilpotent
orbit $\CO\in \Nil_{\bfG}(\g)$, then $\Thetav(\Pi)$ is
$\oliftc_{\bfV,\bfV'}(\CO)$-bounded.
\end{prop}

\subsection{Descent and lift of nilpotent orbits}
\label{DesLif}

Let
\begin{equation*}
\bfWo := \{T \in \bfW | \ T \text{ is an injective map from $\bfV$ to $\bfV'$}\}.
\end{equation*}
Clearly $\bfWo\neq \emptyset$ only if $\dim \bfV\leq \dim \bfV'$.

Suppose $\bfee\in \CO \in \Nil_{\bfG}(\g)$ and
$\bfee'\in \CO' \in \Nil_{\bfG'}(\g')$.  We call $\bfee$ (resp. $\CO$) a
descent of $\bfee'$ (resp. $\CO'$), if there exists  $T\in \bfWo$ such that
$$
\MM(T) = \bfee\quad\textrm{and}\quad \MM'(T) = \bfee'.
$$
For a rational dual pair $(G,G')\subseteq \Sp(W)$, put $$\Xo := \bfW^\circ \cap \CX.$$
Suppose $X\in \sO \in \Nil_{\bfK}(\p)$ and
$X'\in \sO' \in \Nil_{\bfK'}(\p')$.  We call $X$ (resp. $\sO$) a
descent of $X'$ (resp. $\sO'$), if there exists
$T\in \Xo$ such that
$$
\MMP(T) = X\quad\textrm{and}\quad \MMP'(T) = X'.
$$
In both cases, we will say that $T$ realizes the descent, and $\CO'$ (resp. $\sO'$) is the lift of $\CO$ (resp. $\sO$).
We will write
\begin{equation}
\label{eq:descents}
\CO = \DDc(\CO')=\DDc_{\bfV', \bfV}(\CO')\quad \textrm{and}\quad  \sO = \DD(\sO')=\DD_{\bfV', \bfV}(\sO').
\end{equation}
By using explicit formulas in \cite{KP82,DKPC} (for complex dual pairs) and
\cite[Lemma~14]{Oh} (for rational dual pairs), we have the following key property:
\begin{equation*}
\MM(\MM'^{-1}(\overline{\CO})) = \overline{\CO'} \quad \text{and} \quad
\MMP(\MMP'^{-1}(\overline{\sO})) = \overline{\sO'},
\end{equation*}
where ``$\;\overline{\phantom{m}}\;$'' means taking Zariski closure.  (The notion of lift is thus stronger than the notion of theta lift for a complex nilpotent orbit.)

Given $T\in \Xo$ which realizes the descent from  $X' = \MMP'(T)\in \sO'\in \Nil_{\bfK'}(\p')$
  to   $X = \MMP(T)\in \sO\in \Nil_{\bfK}(\p)$, we denote the respective isotropy subgroups by
  \[
    \bfS_T :=
    \Stab_{\bfK\times \bfK'}(T), \quad  \bfK_X:= \Stab_{\bfK}(X)\quad \text{and}\quad \bfK'_{X'}
    := \Stab_{\bfK'}(X').
  \]
  Then there is a unique homomorphism
  \begin{equation}
    \label{eq:alpha}
    \alpha =\alpha _T\colon \bfK'_{X'} \rightarrow \bfK_{X}
  \end{equation}
  such that $\bfS_T$ is the graph of $\alpha$:
  \[
    \bfS_T = \{(\alpha(k'),k')\in \bfK_X\times \bfK'_{X'}|\ k'\in \bfK'_{X'}\}.\]
The homomorphism $\alpha$ is uniquely determined by the requirement that
\[
T(\alpha(k')(v)) = k'(T(v))\quad \textrm{ for all }v\in \bfV, \, k'\in \bfK'_{X'}.
\]

\subsection{Lift of algebraic vector bundles}
\label{subsec:lift.AC}

Attached to a dual pair $(G(V),G(V'))$, there is a distinguished genuine character
  $\mktvvp$ of $\wt{\bfK}\times\wt{\bfK'}$ arising from the oscillator module
  $\Omega _{V,V'}$, as in \cite[Section 2.7]{MSZ}.

Recall the notations from Section \ref{subsec:AC}. Thus $\sO\in \CO \cap \p$ is a $\bfK$-orbit, and we identify $\CK_{\sO}(\wtbfK)$ with $\CR(\wtbfK_{X})$, the Grothendieck group of the category of genuine algebraic representations of the stabilizer group $\wtbfK_X$.

  Let $\rho$ be a genuine algebraic representation  of $\wtbfK_{X}$. Then the representation $\mktvvp|_{\wt{\bfK}_{X}}\otimes \rho$ of $\wtbfK_{X}$ descends  to a representation of $\bfK_{X}$.  Define
  \begin{equation*}
    \dliftv_{T}(\rho):= \mktvvp|_{\wt{\bfK'}_{X'}} \otimes (\mktvvp|_{\wt{\bfK}_{X}}\otimes \rho)\circ \alpha,
  \end{equation*}
  which is a genuine algebraic representation of $\wt{\bfK'}_{X'}$.

  Clearly $\dliftv_T$ induces a homomorphism from
  $\CR(\wtbfK_{X})$ to $\CR(\wtbfK'_{X'})$.
  In view of \eqref{idenkr}, we thus have a homomorphism
  \begin{equation*}
    \xymatrix{
      \dliftv_{\sO,\sO'}\colon \CK_{\sO}(\wt{\bfK}) \ar[r]&
      \CK_{\sO'}(\wt{\bf{K'}}).
    }
  \end{equation*}
  This is independent of the choice of $T$.

  Suppose $\CO\in\Nil_{\bfG}(\g)$, $\CO'\in \NilGCp$ and $\CO =\DDc(\CO')$. Using the decomposition in \eqref{eq:dec.KO}, we define a homomorphism
  \begin{equation}\label{eq:DS.chc}
    \xymatrix{
      \dliftv_{\CO,\CO'} := \displaystyle\sum_{\substack{\sO'\subset \CO'\cap \p'\\ \sO =
          \DD(\sO') \subset \p}}\dliftv_{\sO,\sO'}\colon \CK_{\CO}(\wt{\bfK}) \ar[r]&
      \CK_{\CO'}(\wt{\bf{K'}})
    }
  \end{equation}
  where the summation
  is over all pairs  $(\sO, \sO')$ such that $\sO\subset \p$ is the descent of
  $\sO'\subset \CO'\cap \p'$.

\subsection{An upper bound and an equality of associated characters}
\label{subsec:UBEQ}

We have an upper bound of associated characters in the descent situation.

Recall the natural partial order $\succeq$ on $\CK_{\CO}(\wt{\bfK})$ and
  $\CK_{\sO}(\wt{\bfK})$ defined as follows. We say $c_1\succeq c_2$ (or $c_2\preceq c_1$) if $c_1-c_2$ is represented by a
  $\wt{\bfK}$-equivariant coherent sheaf.

\begin{prop}[{\cite[Theorem~4.3]{MSZ}}]
\label{prop:GDS.AC}
  Let $\CO\in \NilGC$, $\CO'\in \NilGCp$ and $\CO$ is a descent of $\CO'$.
  Then for every genuine $\CO$-bounded $(\g,\wtK)$-module $\Pi$ of finite length,  $\Thetav_{V,V'}(\Pi)$ is $\CO'$-bounded and
  \[
    \Ch_{\CO'}(\Thetav _{V,V'}(\Pi))\preceq \dliftv_{\CO,\CO'}(\Ch_{\CO}(\Pi)).
  \]
\end{prop}

\begin{rmk} We may also define the notion of generalized descent in the setting of complex or rational dual pairs \cite[Section 2.6]{MSZ}. Proposition \ref{prop:GDS.AC} extends to a more general situation, where $\CO'$ is ``good for generalized descent''. See \cite[Section 4.1]{MSZ}.
\end{rmk}

The proofs of Propositions \ref{cor:Cbound} and \ref{prop:GDS.AC} are similar and are largely geometric. There are natural good filtrations on $\Pi$ and $\Thetav_{V,V'}(\Pi)$, which are generated by the minimal degree $\wtK$-types and $\wtK'$-types, respectively. After twisting by $\mktvvp$, their graded spaces may be viewed as a $\bfK$-equivariant coherent sheaf on $\p$ and $\bfK'$-equivariant coherent sheaf on $\p'$, respectively. Thus it boils down to understanding this lifting of equivariant coherent sheaves induced by the double fiberation of moment maps, and this can be readily analyzed via algebraic-geometric techniques. See \cite[Section 4.1]{MSZ} for details.

The final theorem we will discuss is more precise and it is concerned with the associate character of a quotient of $\Thetav_{V,V'}(\pi)$, where $\pi $ is a certain genuine irreducible Casselman-Wallach representation of $\wtG$. Note that by Proposition \ref{prop:GDS.AC}, we will have an upper bound on the full maximal quotient $\Thetav_{V,V'}(\pi)$.

We will review the notion of convergent range in the setting of local theta correspondence (\cite[Sections 3.1 and 3.2]{MSZ}).
First recall that a function on a real reductive group $G$ is called $\nu$-bounded if it is bounded by the $\nu$-th power of the Harish-Chandra's $\Xi$-function $\Xi_G$ on $G$, up to a function of logarithmetic growth. From the well-known properties of $\Xi_G$, we see that every $\nu$-bounded continuous function on $G$ is integrable, for $\nu >2$.
One may also define the notion of $\nu$-boundedness for a Casselman-Wallach representation in the obvious way, by using its matrix coefficients. Under this terminology, tempered representations are then $1$-bounded.

Fix a type I dual pair $(G,G')=(G(V),G(V'))$ as in Section \ref{sec:DPreview}, where $V$ is an $\epsilon$-Hermitian space and $V'$ is an $\epsilon'$-Hermitian space.
Define
\begin{equation}
\dim_{\rF}^{\circ} V = \dim_{\rF} V -2\frac{d_1}{d},
\end{equation}
where $d=\dim_{\rF}D$, and $d_1=\dim_{\rF}\{t\in D|\bar{t}=\epsilon t\}$. Here the superscript bar denotes the standard involutive anti-automorphism of $D$.

We have the following estimate of matrix coefficients. See \cite[Lemma 3.8]{MSZ} or \cite[Theorem 3.2]{Li89}.
\begin{lem}
 The representation $\omega_{V,V'}|_{\cover{G}}$ is $\frac{\dim_{\rF} V'}{\dim_{\rF}^{\circ}V}$-bounded.
\end{lem}

We thus make the following

\begin{defn}\label{defn:CR} Assume that $\dim_{\rF}^{\circ} V >0$. Let $\pi$ be a genuine Casselman-Wallach representation of $\cover{G}$.
   The pair $(\pi, V')$ (or $(\pi, G')$) is said to be in the \emph{convergent range} if $\pi$ is $\nu_\pi$-bounded for some $\nu_\pi>2-\frac{\dim_{\rF} V'}{\dim_{\rF}^{\circ}V}$.
\end{defn}

Suppose that $(\pi, V')$ is in the convergent range. This ensures that the following integral is convergent:
\begin{equation}\label{intpios}
  \xymatrix@R=0em@C=3em{
   (\pi\widehat \otimes \omega_{V,V'})\times (\pi^\vee \widehat \otimes
   \omega_{V,V'}^\vee)
   \ar[r]&\C,\hspace*{5em}\\
   (u,v)\ar@{|->}[r] &\int_{\wt{G'}} <g\cdot u, v> \mathrm{d} g.
   }
\end{equation}
Define
\begin{equation*}
\label{thetab0}
  \Thetab_{V,V'}(\pi):=\frac{\pi\widehat \otimes \omega_{V,V'}}{\textrm{the left kernel of \eqref{intpios}}}.
\end{equation*}
This is a Casselman-Wallach representation of $\wt{G'}$, since it is a quotient of $(\pi\widehat \otimes \omega_{V,V'})_G$ (the Hausdorff coinvariant space, which is the full theta lift of $\pi ^{\vee}$).

As usual we parameterize a nilpotent $\bfG'$-orbit $\CO'$ in $\g'$ by its Young diagram ${\bf D}(\CO')$, but labelled by its column partition $[c_0,c_1,\cdots, c_k]$, where
$c_0\geq c_1\geq \cdots \geq c_k > 0$, and $\sum_{l=0}^k c_l=\dim \bfV'$ (and with some additional constraints on $c_l$'s depending on $\g'$).

\begin{thm}[{\cite[Theorem~4.4]{MSZ}}]
\label{asso}
Let  $\CO\in \Nil_{\bfG}(\g)$ and  $\CO'\in \Nil_{\bfG'}(\g')$  such that
    $\CO$ is the descent of $\CO'$.  Write ${\bf D}(\CO') =
    [c_0,c_1, \cdots, c_k]$ ($k\geq 1$).  Assume that
   $c_0> c_1$ when $G$ is a real symplectic group.  Then for every $\CO$-bounded irreducible Casselman-Wallach representation $\pi$ of $\wtG$ such that $(\pi, V')$ is in
  the convergent range, $\Thetab_{V,V'}(\pi)$ is $\CO'$-bounded and
  \begin{equation*}
  \label{eq:LCh}
    \Ch_{\CO'}(\Thetab_{V,V'}(\pi))= \dliftv_{\CO,\CO'} (\Ch_{\CO}(\pi)).
  \end{equation*}
\end{thm}

\begin{rmk} (a). Theorem \ref{asso} is the most crucial technical ingredient of \cite{MSZ} on the construction of unipotent representations of real classical groups. When $(V, V')$ is in the stable range, the result was proved for unitary representations $\pi $ in \cite{LM}. (b) Note that in \cite{MSZ} the authors
only consider the cases of real (or quaternionic) orthogonal and symplectic groups. The case of unitary groups is proved by a similar method.
\end{rmk}

We outline the basic ideas in the proof of Theorem \ref{asso} for the cases of real (or quaternionic) orthogonal and symplectic groups and offer brief remarks for the case of unitary groups.

First observe that the Harish-Chandra module of $\Thetab_{V,V'}(\pi)$ is isomorphic to a quotient of
$\Thetav(\pi^{\mathrm{al}})$, where $\pi^{\mathrm{al}}$ denotes the Harish-Chandra module of $\pi$. Thus Proposition \ref{prop:GDS.AC} implies that $\Thetab_{V,V'}(\pi)$ is $\CO'$-bounded and
\begin{equation}
\label{boundch}
\Ch_{\CO'}(\Thetab_{V,V'}(\pi))\preceq \dliftv_{\CO,\CO'}(\Ch_{\CO}(\pi)).
\end{equation}

In order to show that the equality in \eqref{boundch} is actually achieved, we will apply another (appropriate) theta lifting by integration. This idea is due to He \cite{He}, who used it to show the nonvanishing of a certain theta lift. We choose a split space $U$, of the same type as $V$, which is roughly twice the size of $V'$. The reader is referred to \cite[Section 3.5]{MSZ} for the precise description of $U$. The space $V$ is then realized as a non-degenerate subspace of $U$ and we write $V^\perp$ for the orthogonal complement of $V$ in $U$. We will consider the two-step theta lifting by integration $\Thetab_{V',V^\perp}(\Thetab_{V,V'}(\pi))$.

The Rallis quotient (the maximal quotient of the trivial representation of $G(V')$ with respect to the dual pair $(U, V')$ will now sit inside a degenerate principal series representation of $\cover{G(U)}$ induced from a certain character of its Siegel parabolic subgroup. Because of our choice of $U$, the corresponding degenerate principal series is on the unitary axis when $G$ is a real orthogonal or a real symplectic group, and is induced from $\det^{\frac{1}{2}}$ (resp. $\det^{-\frac{1}{2}}$) when $G$ is a quaternionic orthogonal group (resp. a quaternionic symplectic group). Here $\det$ denotes the reduced norm of an appropriate general linear group in the quaternionic case. Note that the structure of afore-mentioned degenerate principal series of $\cover{G(U)}$ relative to Rallis quotients is well-known by the work of Kudla-Rallis \cite{KR1}, Lee-Zhu \cite{LZ2,LZ3} (for real orthogonal and real symplectic groups) and Yamana \cite{Ya} (for quaternionic orthogonal and quaternionic symplectic groups). (For unitary groups, the corresponding results are in \cite{LZ1}.)

One shows that $(\Thetab_{V, V'}(\pi), V^\perp)$ is in the convergent range. The decisive observation (\cite[Lemma 4.5]{MSZ} and \cite[Proposition 3.16]{MSZ})
which enables us to conclude the desired equality of associated characters is as follows: via a variant of the well-known doubling method (c.f. \cite{Ra,MVW}), the two-step theta lifting $\Thetab_{V',V^\perp}(\Thetab_{V,V'}(\pi))$ is inside a certain degenerate principal series representation of
$\cover{G (V^\perp)}$ which is intimately related to the aforementioned degenerate principal series of $\cover{G(U)}$ and in a completely explicit way. A result of Barbasch \cite{Ba00} then allows us to determine the associate cycle of $\Thetab_{V',V^\perp}(\Thetab_{V,V'}(\pi))$.

On the other hand, there is an upper bound of the associate character of $\Thetab_{V',V^\perp}(\Thetab_{V,V'}(\pi))$ (by applying Proposition \ref{prop:GDS.AC} twice). By comparing the associated cycles, one concludes that there is no other choice except to have the desired equality in \eqref{boundch}. This comparison is also the step where one needs to do a case-by-case check for real orthogonal and real symplectic groups, and for quaternionic orthogonal and quaternionic symplectic groups, and similarly for unitary groups.


\begin{thebibliography}{999}
\frenchspacing

\bibitem{BS}
Y. Bao and B. Sun, \emph{Coincidence of algebraic and smooth theta correspondences},
Represent. Theory 21, (2017), 458--466.

\bibitem{Ba00} D. Barbasch, \emph{Orbital integrals of nilpotent orbits}, The mathematical legacy of {H}arish-{C}handra, Proc. Sympos. Pure Math., V68, 97--110,
Amer. Math. Soc., Providence, RI, 2000.

\bibitem{BV80} D. Barbasch and D. A. Vogan Jr., \emph{The local structure of characters,} J. Funct. Anal. 37, (1980), no. 1, 27--55.

\bibitem{Ca89}
W. Casselman, \emph{Canonical extensions of Harish-Chandra modules to representations of $G$}, Canad. J. Math. 41, (1989), 385--438.

\bibitem{Cl} F. du Cloux, \emph{Sur les reprsentations diffrentiables des
groupes de Lie algbriques}, Ann. Sci. Ecole Norm. Sup. 24, (1991), no. 3, 257--318.

\bibitem{CM92} D. H. Collingwood  and W. M. McGovern,  \emph{Nilpotent Orbits in Semisimple Lie Algebras,} Van Nostrand Reinhold Mathematics Series, 1992.

\bibitem{DKPC}
A. Daszkiewicz, W. Kraskiewicz and T. Przebinda, \emph{Nilpotent orbits and complex dual pairs}, J. Algebra 190, no. 2, (1997), 518--539.

\bibitem{DKP2} A. Daszkiewicz, W. Kraskiewicz and T. Przebinda, \emph{Dual pairs and Kostant-Sekiguchi correspondence. II. Classification of nilpotent elements},
Cent. Eur. J. Math. 3, no. 3, (2005), 430--474.

\bibitem{Fu95}
M. Furusawa, \emph{On the theta lift from $\SO_{2n+1}$ to $\cover{\Sp}_n$}, J. Reine Angew. Math. 466, (1995), 87–110.

\bibitem{Ga} W.T. Gan, \emph{Periods and theta correspondence}, this volume.

\bibitem{GGP}
W.T. Gan, B.H. Gross, and D. Prasad, \textit{Symplectic local root
numbers, central critical $L$-values, and restriction problems in
the representation theory of classical groups}, Asterisque 346, (2012), 111--170.

\bibitem{GS}
W.T. Gan and B. Sun, \textit{The Howe duality conjecture: quaternionic case}, in
Representation Theory, Number Theory, and Invariant Theory, In Honor of Roger Howe.  J. Cogdell et al. (eds.), Progress in Math. 323, 175--192, Birkh\"auser, 2017.

\bibitem{GT}
W.T. Gan and S. Takeda, \emph{A proof of the Howe duality conjecture}, J. Amer. Math. Soc. 2, (2016), 473--493.

\bibitem{GJS}
D. Ginzburg, D. Jiang and D. Soudry, \emph{Poles of L-functions and theta liftings for orthogonal groups, II}, in ``On certain
L-functions'', 141–-158, Clay Math. Proc., 13, Amer. Math. Soc., Providence, RI, 2011.

\bibitem{GZ}
R. Gomez and C.-B. Zhu, \emph{Local theta lifting of generalized Whittaker models associated to nilpotent orbits}, Geom. Funct. Ana. 24, (2014), 796--853.

\bibitem{HC78} Harish-Chandra, \emph{Admissible invariant distributions on reductive p-adic groups,} Queen's Papers in Pure and Applied Math. 48, (1978), 281--347.

\bibitem{He}
H. He, \emph{Unipotent representations and quantum induction}, preprint (2002), arXiv:math/0210372.

\bibitem{Ho79}
R. Howe, \emph{$\theta$-series and invariant theory}, in
Automorphic Forms, Representations and $L$-functions, Proc. Symp.
Pure Math. 33, (1979), 275--285.

\bibitem{HoWF}
R. Howe, \emph{Wave front sets of representations of Lie groups}, in Automorphic forms, representation theory and arithmetic (Bombay, 1979), pp. 117--140, Tata Inst. Fund. Res. Studies in Math., 10, Tata Inst. Fundamental Res., Bombay, 1981.

\bibitem{Ho89}
R. Howe, \emph{Transcending classical invariant theory}, J. Amer.
Math. Soc. 2, (1989), 535--552.

\bibitem{Ji07} D. Jiang, \emph{Periods of automorphic forms}, Proceedings of the International Conference on Complex Geometry and Related Fields,
Studies in Advanced Mathematics, Vol 39, 2007, pp. 125--148, American Mathematical Society and International Press.

\bibitem{Ka85} N. Kawanaka, ``Generalized Gelfand-Graev representations and Ennola duality'', in \emph{Algebraic groups and related topics,} Advanced Studies in Pure Math. 6, (1985), pp. 175--206.

\bibitem{Ko78} B. Kostant, \emph{On Whittaker vectors and representation theory,} Invent. math. 48, (1978), 101--184.

\bibitem{KP82} H. Kraft and C. Procesi, \emph{On the geometry of conjugate classes in classical groups,} Comment. Math. Helvetici 57, (1982), 539--602.

\bibitem{KV96}
J. A. C. Kolk and V. S. Varadarajan, \emph{On the transverse symbol of vectorial distributions and some applications to harmonic analysis,} Indag. Math. (N.S.) 7, no.\ 1, (1996), 67--96.

\bibitem{KR1}
S. S. Kudla and S. Rallis, \emph{Degenerate principal series and invariant distributions}, Israel J. Math. 69, (1990), 25--45.

\bibitem{KR2}
S. S. Kudla and S. Rallis, \emph{On first occurrence in the local
theta correspondence}, in ``Automorphic Representations,
$L$-functions and Applications: Progress and Prospects'', Ohio State
Univ. Math. Res. Inst. Publ., vol. 11, 273--308. de Gruyter, Berlin, 2005.

\bibitem{La} R. P. Langlands, \emph{On the classification of irreducible representations of real algebraic groups}, Representation theory and harmonic analysis on semisimple Lie groups, 101–-170,
Math. Surveys Monogr., 31, Amer. Math. Soc., Providence, RI, 1989.

\bibitem{Li89}
J.-S. Li, \emph{Singular unitary representations of classical groups}, Invent. math. 97, (1989), 237--255.

\bibitem{LZ1}
S. T. Lee and C.-B. Zhu, \emph{Degenerate principal series and local theta correspondence}, Trans. Amer. Math. Soc. 350, No. 12 (1998), 5017-5046.

\bibitem{LZ2}
S. T. Lee and C.-B. Zhu, \emph{Degenerate principal series and local theta correspondence II}, Israel J. Math. 100, (1997), 29--59.

\bibitem{LZ3}
S. T. Lee and C.-B. Zhu, \emph{Degenerate principal series of metaplectic groups and Howe correspondence}, in Automorphic Representations and L-Functions, D. Prasad at al. (eds.),  379--408, Tata Institute of Fundamental Research, India, 2013.

\bibitem{LM}
H. Y. Loke and J.-J. Ma, \emph{Invariants and $K$-spectrums of local theta lifts}, Compositio Math. 151, no. 1, (2015), 179--206.

\bibitem{Ma92} H. Matumoto, \emph {$C^{-\infty}$-Whittaker vectors corresponding to a principal nilpotent orbit of a real reductive linear Lie group, and wave front sets,} Compositio Math. 82, (1992), 189--244.

\bibitem{MSZ}
J.-J. Ma, B. Sun and C.-B. Zhu, \emph{Unipotent representations of real classical groups}, preprint (2017), arXiv:1712.05552.

\bibitem{Mi}
A. Minguez, \emph{Correspondance de Howe explicite: paires duales de type II}, Ann. Sci. Ecole Norm. Sup. 41, (2008), 717–-741.


\bibitem{Mo98} C. M\oe{}glin, \emph{Correspondance de Howe et front d'onde}, Adv. Math. 133, no. 2, (1998), 224--285.


\bibitem{MVW}
C. M{\oe}glin, M.-F. Vign\'eras and J.-L. Waldspurger, \emph{Correspondances de Howe sur un corps $p$-adique}, Lecture Notes in Mathematics 1291, Springer, 1987.

\bibitem{MW87} C. M\oe{}glin and J.-L. Waldspurger, \emph{Mod\`eles de Whittaker d\'eg\'en\'er\'es pour des groupes $p$-adiques}, Math. Z. 196, (1987), 427--452.

\bibitem{NPS73} M. E. Novodvorskii and I. Piatetski-Shapiro, \emph{Generalized Bessel models for a symplectic group of rank $2$}, (Russian) Mat.
Sb. (N.S.)  90 (132), (1973), 246--256.

\bibitem{NOTYK}
K. Nishiyama, H. Ochiai, K. Taniguchi, H. Yamashita and S. Kato, \emph{Nilpotent orbits, associated cycles and Whittaker models for highest weight representations},
Ast\'erisque 273, (2001), 1--163.

\bibitem{NZ}
K. Nishiyama and C.-B. Zhu, \emph{Theta lifting of unitary lowest weight modules and their associated cycles}, Duke Math. J. 125, no. 3, (2004), 415--465.

\bibitem{Oh}
T. Ohta, \emph{The closures of nilpotent orbits in the classical symmetric
    pairs and their singularities}, Tohoku Math. J. 43, no. 2, (1991), 161--211.

\bibitem{Pz}
T. Przebinda, \emph{Characters, dual pairs, and unitary representations}, Duke Math. J. 69, no. 3, (1993), 547--594.

\bibitem{Ra}
S. Rallis, \emph{On the Howe duality conjecture}, Compositio Math. 51, (1984), 333--399.

\bibitem{Ro95} W. Rossmann, \emph{Picard-Lefschetz theory for the coadjoint quotient of a semisimple Lie algebra,} Invent. Math. 121, (1995), no. 3, 531--578.

\bibitem{Se87} J. Sekiguchi, \emph{Remarks on real nilpotent orbits of a symmetric pair,} J. Math. Soc. Japan 39, (1987), no. 1, 127--138.

\bibitem{Sh74} J. A. Shalika, \emph{The multiplicity one theorem for $\GL_{n}$,} Ann. of Math. 100, (1974), 171--193.

\bibitem{SV00} W. Schmid and K. Vilonen, \emph{Characteristic cycles and wave front cycles of representations of reductive Lie groups,} Ann. of Math. 151, (2000), no. 3, 1071--1118.

\bibitem{SZ}
B. Sun and C.-B. Zhu, \emph{Conservation relations for local theta correspondence}, J. Amer. Math. Soc. 28, (2015), 939--983.

\bibitem{SZ2}
B. Sun and C.-B. Zhu, \emph{On the conservation conjectures of Kudla and Rallis}, in Representation Theory, Number Theory, and Invariant Theory, In Honor of Roger Howe, J. Cogdell et al. (eds.),  Progress in Math. 323, 587--602, Birkh\"auser, 2017.

\bibitem{Vo78} D. A. Vogan Jr., \emph{Gelfand-Kirillov dimension for Harish-Chandra modules,} Invent. Math. 48, (1978), 75--98.

\bibitem{Vo79} D. A. Vogan Jr., \emph{The algebraic structure of the representation of semisimple Lie groups. I.,} Ann. of Math. (2) 109 (1979), no. 1, 1–-60.

\bibitem{Vo89}
D. A. Vogan, \emph{Associated varieties and unipotent representations}, In: Barker, W., Sally, P. (Eds.) Harmonic Analysis on Reductive Groups (Bowdoin College, 1989). Progress in Mathematics, vol 101, 315--388, Birkh\"{a}user (Boston-Basel-Berlin), 1991.

\bibitem{Vo98}
D. A. Vogan, \emph{The method of coadjoint orbits for real reductive groups}, Representation theory of Lie groups (Park City, 1998). IAS/Park City Math. Ser., vol. 8, 179--238, Amer. Math. Soc., 2000.

\bibitem{Wald} J.-L. Waldspurger, \emph{D\'emonstration d'une conjecture de dualit\'e de Howe dans le cas $p$-adique, $p\ne 2$,} Festschrift in honor of I. I. Piatetski-Shapiro on the occasion of his sixtieth birthday, Part I (Ramat Aviv, 1989), 267--324, Israel Math. Conf. Proc., 2, Weizmann, Jerusalem, 1990.

\bibitem{WaJI} N. R. Wallach, \emph{Lie Algebra Cohomology and Holomorphic Continuation of Generalized Jacquet Integrals}, Advanced Studies in Math, vol 14, (1988), 123--151.

\bibitem{Wa1}
N. R. Wallach, \emph{Real reductive groups I}, Academic Press Inc., 1988.

\bibitem{Wa2}
N. R. Wallach, \emph{Real reductive groups II}, Academic Press Inc., 1992.

\bibitem{Weyl}
H. Weyl, \emph{The classical groups: their invariants and representations}, Princeton University Press, 1947.

\bibitem{Weil} A. Weil, \textit{Sur certain group d'operateurs unitaires}, Acta Math. 111, (1964), 143--211.

\bibitem{Ya}
S. Yamana, \emph{Degenerate principal series representations for quaternionic unitary groups}, Israel J. Math. 185, (2011), 77--124.

\bibitem{Ya86} H. Yamashita, \emph{On Whittaker vectors for generalized Gelfand-Graev representations of semisimple Lie groups}, J. Math. Kyoto Univ. 26, no. 2, (1986), 263--298.

\bibitem{Ya01} H. Yamashita, ``Cayley transform and generalized Whittaker models for irreducible highest weight modules'' in \emph{Nilpotent Orbits, Associated Cycles and Whittaker Models for Highest weight Representations,} Ast\'erisque 273, (2001), 81--137.


\end{thebibliography}
\end{document}